\documentclass{amsart}
\usepackage{amsmath}
\usepackage{amsthm}
\usepackage{amssymb}
\usepackage{tikz-cd} 
\usepackage{graphicx}
\usepackage{subcaption}
\usepackage[colorlinks=true, linktocpage=true]{hyperref}

\newtheorem{Def}{Definition}[section]
\newtheorem{prop}[Def]{Proposition}
\newtheorem{theorem}[Def]{Theorem}

\newtheorem{lemma}[Def]{Lemma}
\newtheorem{corollary}[Def]{Corollary}
\newtheorem{remark}[Def]{Remark}

\title[$n$-cylinder square-tiled surfaces and volume of $\mathcal{H}(2g-2)$]{Contribution of $n$-cylinder square-tiled surfaces to Masur--Veech volume of $\mathcal{H}(2g-2)$}
\author[Ivan~Yakovlev]{Ivan Yakovlev}
\address{
LaBRI, Universit\'e de Bordeaux, 351, cours de la Lib\'eration, F-33405 Talence, France
}
\email{ivan.yakovlev@labri.fr}

\begin{document}
\begin{abstract}
    We find the generating function for the contributions of $n$-cylinder square-tiled surfaces to the Masur--Veech volume of $\mathcal{H}(2g-2)$. It is a bivariate generalization of the generating function for the total volumes obtained by Sauvaget via intersection theory. Our approach is, however, purely combinatorial. It relies on the study of counting functions for certain families of metric ribbon graphs. Their top-degree terms are polynomials, whose (normalized) coefficients are cardinalities of certain families of metric plane trees. These polynomials are analogues of Kontsevich polynomials that appear as part of his proof of Witten’s conjecture.
\end{abstract}

\maketitle

\tableofcontents

\section{Introduction}
    \subsection{Masur--Veech volumes}
    
    For $g\geq 1$ and a tuple $k=(k_1,\ldots,k_s)$ of $s \geq 1$ non-negative integers with sum $2g-2$, let  $\mathcal{H}(k)$ be the moduli space of tuples $(X,x_1,\ldots,x_s,\omega)$, where $X$ is a Riemann surface of genus $g$, $x_1,\ldots,x_s$ are distinct labeled points of $X$ and $\omega$ is a non-zero Abelian differential on $X$ with zeros of multiplicities $k_1,\ldots,k_s$ at points $x_1,\ldots,x_s$, respectively, and no other zeros. Each such space is called a \emph{stratum} of Abelian differentials.
    
    Every stratum $\mathcal{H}(k)$ is a complex orbifold of dimension $2g+s-1$, which admits an atlas of charts to $\mathbb{C}^{2g+s-1}$ with transition functions given by integer matrices (so-called \emph{period coordinates}). In particular, the Lebesgue measures in different charts are compatible, and their pullbacks to $\mathcal{H}(k)$ give the \emph{Masur--Veech measure} $\nu$ on $\mathcal{H}(k)$. The total volume $\nu(\mathcal{H}(k))$ is always infinite. However, if one restricts to the hypersurface $\mathcal{H}_1(k) \subset \mathcal{H}(k)$ defined by $\frac{i}{2} \int_X \omega \wedge \overline{\omega} = 1$, the Masur--Veech measure induces a measure $\nu_1$ on $\mathcal{H}_1(k)$ which is always finite by the independent results of Masur \cite{Ma} and Veech \cite{Ve}. We call $\nu_1(\mathcal{H}_1(k))$ the \emph{Masur--Veech volume} of $\mathcal{H}(k)$ and we denote it by $\operatorname{Vol}(k)$.
    
    The exact computation of Masur--Veech volumes is important for understanding the dynamics of the billiard flow in rational polygons or, more generally, the geodesic flow in translation surfaces. We refer to the surveys \cite{MT02}, \cite{Wri15}, \cite{Wri16}, \cite{Zor06} for more information.
   
    \subsection{Square-tiled surfaces}
    
    Computation of the Masur--Veech volumes can be reformulated as an asymptotic enumeration problem for \emph{square-tiled surfaces}, which are a special kind of quadrangulations (the latter are well studied in the theory of combinatorial maps). This approach appears in the paper \cite{Zor02}, to which we refer the reader for details. We start with the construction of square-tiled surfaces.
    
    Consider a finite collection of oriented euclidean unit squares with sides labeled as top, bottom, left and right. Identify the sides of these squares by isometries respecting the orientation, gluing top sides to bottom sides, left sides to right sides, to get an oriented closed surface $S$. The standard complex structure and the standard 1-forms $dz$ on the squares are compatible with the identifications and endow $S$ with a complex structure and a non-zero Abelian differential. Labeling the zeros of the differential, we get a point $(X,x_1,\ldots,x_s,\omega)$ in some stratum $\mathcal{H}(k)$, which we call a square-tiled surface. Examples of square-tiled surfaces are given in the first row of Figure \ref{fig:examples}.
    
    To recover $k$ from the gluing of the squares, first note that the gluing rules imply that the number of squares around every vertex is a multiple of 4. A zero of order $k_i$ of $\omega$ is then represented by a vertex with $4(k_i+1)$ incident squares. The rest of the vertices have 4 incident squares.
    
    Let $\mathcal{ST}(\mathcal{H}(k), N)$ denote the set of square-tiled surfaces in $\mathcal{H}(k)$ with at most $N$ squares and let $d=2g+s-1$ be the complex dimension of $\mathcal{H}(k)$. Then
    \begin{equation}
        \label{eq:volume_via_counting}
        \operatorname{Vol}(k) = 2d \cdot \lim_{N \rightarrow +\infty} \frac{|\mathcal{ST}(\mathcal{H}(k), N)|}{N^d}.
    \end{equation}
    
    \subsection{Cylinder decomposition and main theorem}
    \label{subsec:cylinder_dec_main_theorem}

    Note that for a square-tiled surface, the standard flat metric of the squares induces a singular flat metric on the surface: a vertex with $4(k_i+1)$ incident squares is a \emph{conical singularity} of angle $2\pi(k_i+1)$, while the rest of the vertices (as well as the interior points of the squares and their edges) are regular ``flat points'' -- the total angle around them is $2\pi$.
    
    Consider a square-tiled surface $S$ with its singular flat metric. Since the number of squares is finite, every horizontal side of every square is either a part of a geodesic joining conical singularities (these can coincide), which we call a \emph{horizontal saddle connection}, or a part of a simple closed geodesic not passing through any singularities. Let $G_S$ be the union of all conical singularities and horizontal saddle connections of $S$. Consider the complement $S \setminus G_S$. The closure in $S$ of any connected component of $S \setminus G_S$ carries a non-singular flat metric with geodesic boundary, so, by Gauss-Bonnet theorem, it is a (square-tiled) cylinder. Let $n$ be the total number of cylinders in $S \setminus G_S$. 
    
    In analogy with (\ref{eq:volume_via_counting}), we can now consider the contribution of $n$-cylinder square-tiled surfaces to the Masur--Veech volume of $\mathcal{H}(k)$:
    \begin{equation}
        \label{eq:contribution_via_counting}
        \operatorname{Vol}_n(k) = 2d \cdot \lim_{N \rightarrow +\infty} \frac{|\mathcal{ST}_n(\mathcal{H}(k), N)|}{N^d},
    \end{equation}
    where $\mathcal{ST}_n(\mathcal{H}(k), N)$ is the set of square-tiled surfaces in $\mathcal{H}(k)$ with $n$ cylinders and at most $N$ squares. The existence of the limit in (\ref{eq:contribution_via_counting}) is not obvious, see Section 1.1 in \cite{DGZZ-1_cylinder} and the references therein. For the special case $k=(2g-2)$ we independently prove the existence of the limit as part of our proof of the main Theorem~\ref{thm:main_theorem} (see the proof of Proposition~\ref{prop:explicit_formulas_contributions}). Note that for a square-tiled surface of genus $g$ the number of cylinders $n$ is at most $g$. Then, clearly, $\operatorname{Vol}(k) = \sum_{n=1}^g \operatorname{Vol}_n(k)$.
    
    From now on we restrict our attention to the \emph{minimal stratum} $\mathcal{H}(2g-2)$.  The main result of this paper is the following
    \begin{theorem}
    \label{thm:main_theorem}
        The contribution $\operatorname{Vol}_n(2g-2)$ of $n$-cylinder square-tiled surfaces to the volume of the minimal stratum $\mathcal{H}(2g-2)$ is equal to $\frac{2(2\pi)^{2g}}{(2g-1)!}a_{g,n}$, where the numbers $a_{g,n} \in \mathbb{Q}$, and whose bivariate generating function
        \[\mathcal{C}(t,u) = 1 + \sum_{g\geq 1} \left(\sum_{n=1}^g a_{g,n} u^n \right) (2g-1)t^{2g}\]
        satisfies for all $g\geq 0$
        \begin{equation}
            \label{eq:relation_bivariate}
            \frac{1}{(2g)!} [t^{2g}] \mathcal{C}(t,u)^{2g} = [t^{2g}] \left(\frac{t/2}{\sin(t/2)}\right)^u,
        \end{equation}
        where $[t^{2g}]$ stands for the extraction of the coefficient of the corresponding monomial.
    \end{theorem}
    Using Lagrange inversion, (\ref{eq:relation_bivariate}) can be rewritten equivalently as 
    \begin{equation*}
        \mathcal{C}(t,u)= \frac{t}{Q^{-1}(t,u)}, \ Q(t,u) = t\cdot \exp \left( \sum_{k=1}^{\infty} (k-1)! b_k(u) t^k \right),
    \end{equation*}
    where $b_k(u) = [t^k]\left(\frac{t/2}{\sin(t/2)}\right)^u$ and functional inversion is with respect to the variable $t$. In particular, the numbers $a_{g,n}$ can be effectively computed. We present in Table \ref{tab:agn} the values of $a_{g,n}$ for small genera $g$.
    
    \begin{table}[h]
        \centering
        \renewcommand{\arraystretch}{1.5}
        \begin{tabular}{|c|c|c|c|c|}
            \hline
            $g \backslash n$ & 1 & 2 & 3 & 4\\
            \hline
            1 & $\frac{1}{24}$ & & & \\
            \hline
            2 & $\frac{1}{1440}$ & $\frac{1}{1152}$ & & \\
            \hline
            3 & $\frac{1}{7560}$ & $\frac{1}{3840}$ & $\frac{11}{82944}$ & \\
            \hline
            4 & $\frac{1}{13440}$ & $\frac{5197}{29030400}$ & $\frac{3}{20480}$ & $\frac{335}{7962624}$ \\
            \hline
        \end{tabular}
        \caption{Values of the normalized volume contributions $a_{g,n}$ for $g \leq 4$.}
        \label{tab:agn}
    \end{table}
    
    \subsection{Remarks on main theorem and strategy of proof}
    \label{subsec:remarksMainTheorem}
    The particular case $u=1$ of Theorem~\ref{thm:main_theorem}, which gives the generating function for the (normalized) total volumes of $\mathcal{H}(2g-2)$, was obtained by Sauvaget \cite[Theorem 1.6]{Sau} via intersection theory. There the numbers $a_g=\sum_{n=1}^g a_{g,n}$ are shown to be equal to certain intersection numbers. The intersection-theoretic interpretation of the refined numbers $a_{g,n}$ (if exists) is currently unknown to the author.
    
    Theorem~\ref{thm:main_theorem} implies in particular that $\operatorname{Vol}_n(2g-2) \in \mathbb{Q}\pi^{2g}$, a fact which was known for the \emph{total} volumes of all strata since the paper of Eskin and Okounkov \cite{EOk}. Note that this is not true for general strata, as exemplified by the result of Delecroix, Goujard, Zograf and Zorich \cite[Proposition 2.2]{DGZZ-1_cylinder} which states that the contribution of 1-cylinder square-tiled surfaces to the volume of any stratum $\mathcal{H}(k)$ is a rational multiple of $\zeta(d)$ (for odd $d$ it is believed to be algebraically independent of $\pi$). 
    
    A similar result (though, without an explicit equation for the generating function) for the \emph{principal} strata $\mathcal{Q}(1^{4g-4+n},-1^n)$ of \emph{quadratic} differentials  was obtained by Delecroix, Goujard, Zograf and Zorich in \cite[Theorem 1.5]{DGZZ}. There the authors group square-tiled surfaces according to their stable graph. The contribution to the total volume of surfaces with fixed stable graph is then a rational multiple of $\pi^{6g-6+2n}$. We do not precise here the definition of the stable graph of a square-tiled surface, but we note that in the case of the minimal stratum $\mathcal{H}(2g-2)$ grouping by stable graphs is equivalent to grouping by number of cylinders.
    
    In \cite{DGZZ}, the strategy of the proof is to express the corresponding contributions through counting functions for trivalent metric ribbon graphs of fixed genus $g$ and $n$ labeled boundary components of given perimeters. The top-degree terms of these functions turn out to be polynomials, a result appearing (in a different form) as part of Kontsevich's proof \cite{Kon} of Witten's conjecture \cite{Wi} and also (in this form) in the paper of Norbury \cite{Nb} (see more details in Section~\ref{subsec:analogy}).
    
    To prove Theorem~\ref{thm:main_theorem}, we follow the same strategy: we express the contributions $\operatorname{Vol}_n(2g-2)$ through counting functions for a different family of metric ribbon graphs (Proposition~\ref{prop:k-cyl_contribution}), and then we prove (Theorem~\ref{thm:top_term}) that their top-degree terms are actually polynomials (outside of a finite number of hyperplanes). Their coefficients have a combinatorial interpretation as cardinalities of certain families of metric plane trees, which allows to deduce an equation for their generating function (Theorem~\ref{thm:polys}). Theorem~\ref{thm:main_theorem} is then deduced from Theorem~\ref{thm:polys} in Section~\ref{subsec:proof_main_theorem}.
    
    \subsection{Bibliographic notes}
    
    Zorich \cite{Zor02} computed the Masur--Veech volumes for small genera. Eskin and Okounkov \cite{EOk}, using the representation theory of the symmetric group, proposed an algorithm for the computation of the volumes. They were also able to show that $\operatorname{Vol}(k) \in \mathbb{Q}\pi^{2g}$. Later, explicit generating functions were given for the volumes of \emph{minimal} strata $\mathcal{H}(2g-2)$ by Sauvaget \cite{Sau} (via intersection theory), and for the volumes of \emph{principal} strata $\mathcal{H}(1,\ldots,1)$ by Chen, M\"{o}ller and Zagier \cite{CMZ} (via quasimodularity of certain generating functions). Finally, Chen, M\"{o}ller, Sauvaget and Zagier \cite{CMSZ} produced a recursion for the Masur--Veech volumes of general strata.
    
    Delecroix, Goujard, Zograf and Zorich in \cite{DGZZ} computed the contributions of square-tiled surfaces with fixed stable graph (and even fixed heights of the cylinders) to the volumes of principal strata of \emph{quadratic} differentials. This allowed them in \cite{DGZZ} and \cite{DGZZ-asymptoticGeometry} to study the asymptotic properties of random closed multicurves on surfaces (such as the topological type, the weights and the number of primitive components), via the correspondence between multicurves and square-tiled surfaces with a fixed cylinder decomposition: asymptotic probabilities for multicurves are equal to relative volume contributions of square-tiled surfaces, primitive components correspond to horizontal cylinders and weights of components correspond to heights of the cylinders.

    \subsection{Acknowledgement} I would like to thank my supervisor Vincent Delecroix for posing the problem as well as for the support and useful discussions during the preparation of this paper.

    \section{Counting functions for metric ribbon graphs}
    In Section~\ref{subsec:metric_ribbon_graphs} we introduce the main objects of this paper --- the counting functions $\mathcal{P}^g_{k,l}(L,L')$ for certain families of metric ribbon graphs. In Proposition~\ref{prop:k-cyl_contribution} we reduce the counting of square-tiled surfaces in minimal strata to the study of this counting functions. Then, in Section~\ref{subsec:local_polys} we state the relevant properties of these functions, the proofs of which occupy the rest of the paper. In Section~\ref{subsec:analogy} we briefly discuss the analogy of the top-degree terms of $\mathcal{P}^g_{k,l}(L,L')$ with the Kontsevich polynomials.
    
    \subsection{Metric ribbon graphs and counting of square-tiled surfaces}
    \label{subsec:metric_ribbon_graphs}
    
    A \emph{ribbon graph} is a connected graph (loops and multiple edges are allowed) with a circular ordering of adjacent edges at each vertex. An \emph{isomorphism of ribbon graphs} is a graph isomorphism respecting the circular orderings. The circular orderings allow to construct a tubular neighborhood of the ribbon graph and to view it as an oriented surface with boundary: replace each edge by an oriented ribbon, then glue the ribbons around the vertices according to the circular orderings and respecting the orientations. In particular, a ribbon graph has a well-defined \emph{genus} and \emph{boundary components}. The number of vertices $V$, number of edges $E$, number of boundary components $F$ and the genus $g$ satisfy the Euler's formula $V-E+F=2-2g$. Examples of ribbon graphs are given in the second and the third rows of Figure \ref{fig:examples}. For more details on ribbon graphs and their applications we refer the reader to \cite{LandoZvonkin}.
    
    A \emph{metric} on a ribbon graph $G$ is an assignment of a positive length to every edge of $G$. The metric is \emph{integral} if the lengths of all edges are integers. The \emph{perimeter} of a boundary component is the sum of the lengths of the edges along this boundary component.
    
    For $g \geq 0, k,l \geq 1$ denote by $\mathcal{E}_{g,k,l}$ the set of isomorphism classes of genus $g$ ribbon graphs with one vertex, $k$ black boundary components labeled from 1 to $k$, $l$ white boundary components  labeled from 1 to $l$, such that any two boundary components sharing an edge have opposite colors (isomorphisms must respect the coloring and the labeling of the boundary components). See the second row of Figure \ref{fig:examples}. For $G \in \mathcal{E}_{g,k,l}$, $L=(L_1,\ldots,L_k) \in \mathbb{Z}^k$, $L'=(L'_1,\ldots,L'_l) \in \mathbb{Z}^l$ denote by $\mathcal{P}_G(L,L')$ the number of integral metrics on $G$ giving the black and white boundary components the perimeters $L_1,\ldots,L_k$ and $L'_1,\ldots,L'_l$ respectively. Finally, we introduce the weighted counting function
    \begin{equation}
        \label{eq:counting_function}
        \mathcal{P}^g_{k,l}(L,L') = \sum_{G \in \mathcal{E}_{g,k,l}} \frac{1}{|\operatorname{Aut}(G)|} \cdot \mathcal{P}_G(L,L').
    \end{equation}
    
    \begin{prop}
        \label{prop:k-cyl_contribution}
        The number $|\mathcal{ST}_n(\mathcal{H}(2g-2), N)|$ of $n$-cylinder square-tiled surfaces in $\mathcal{H}(2g-2)$ with at most $N$ squares is equal to
        \begin{equation}
            \label{eq:k-cyl_contribution}
            \frac{1}{n!} \cdot \sum_{\substack{\sum_{i=1}^n h_i L_i \leq N\\ h_i, L_i \in \mathbb{Z}_{>0}}} L_1\cdots L_n \cdot \mathcal{P}
            ^{g-n}_{n,n}(L_1, \ldots, L_n; L_1, \ldots, L_n).
        \end{equation}
    \end{prop}
    \begin{proof}
        Consider a $n$-cylinder square-tiled surface $S$ in $\mathcal{H}(2g-2)$, with cylinders arbitrarily labeled from $1$ to $n$. Recall from Section~\ref{subsec:cylinder_dec_main_theorem} that we denote by $G_S$ the union of all conical singularities and horizontal saddle connections of $S$. Clearly, $G_S$ is a one-vertex graph (whose vertex is the unique conical singularity) with several loops (which are horizontal saddle connections). The embedding of $G_S$ in $S$ gives a circular ordering of incident edges at the vertex. So $G_S$ has a structure of a ribbon graph. For illustration, see the first and the second rows of Figure \ref{fig:examples}.
        
        Moreover, the boundary components of $G_S$ come in two types, depending on whether the bottom or the top sides of the squares are glued to this boundary. We color the first boundaries in black and the second ones in white. Clearly, adjacent boundaries have opposite colors. Each of the $n$ labeled cylinders of $S \setminus G_S$ is glued to one black and one white boundary of $G_S$, so $G_S$ has $n$ black and $n$ white boundaries, which we label by the label of the adjacent cylinder. Finally, gluing a cylinder increases the genus of the surface by 1, so the genus of $G_S$ must be equal to $g-n$. Hence, $G_S \in \mathcal{E}_{g-n,n,n}$.
        
        We now prove the formula (\ref{eq:k-cyl_contribution}). The square-tiled cylinders of $S \setminus G_S$ are uniquely specified by their heights $h_i \in \mathbb{Z}_{>0}$ and their circumferences $L_i \in \mathbb{Z}_{>0}$, $1\leq i \leq n$. The total number of squares in the surface is then $\sum_{i=1}^n h_i L_i$, which gives the inequality condition. Each edge of $G_S$ is endowed with a positive integer length equal to the number of squares glued to either of its sides, and the perimeter of each boundary component of $G_S$ is equal to the circumference of the adjacent cylinder. This gives the term $\mathcal{P}^{g-n}_{n,n}(L_1, \ldots, L_n; L_1, \ldots, L_n)$. The term $L_1\cdots L_n$ comes from the fact that for the cylinder with label $i$, there are $L_i$ different ways to twist it before gluing to $G_S$. Finally, the term $\frac{1}{n!}$ accounts for the arbitrariness of the numbering of the $n$ cylinders.
    \end{proof}
     
     Proposition~\ref{prop:k-cyl_contribution} together with formula (\ref{eq:contribution_via_counting}) reduce the study of the volume contributions $\operatorname{Vol}_n(2g-2)$ to the study of the counting functions $\mathcal{P}^g_{k,l}$.
    
    \begin{figure}
        \centering
        \includegraphics[width=\textwidth]{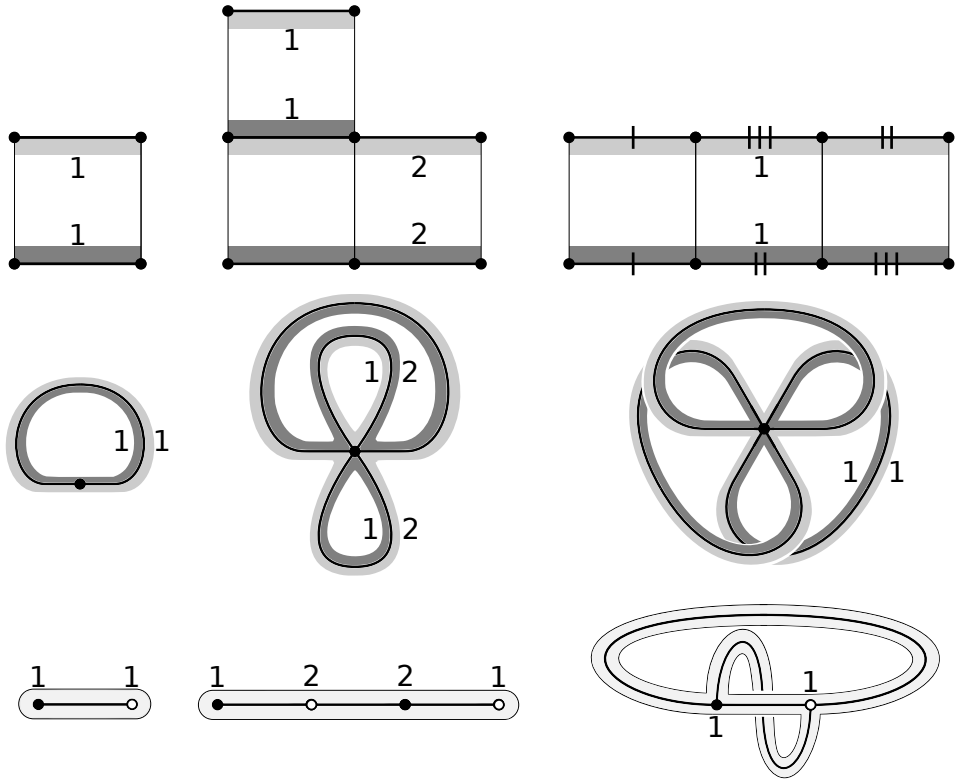}
        \caption{\emph{First row}: examples of square-tiled surfaces $S$ represented as polygons. In the first two polygons, opposite sides should be identified. In the third one, horizontal sides should by identified according to their labels. In each case, all vertices of all squares are identified into a single conical singularity. The first surface belongs to $\mathcal{H}(0)$, second and third -- to $\mathcal{H}(2)$. First and third surfaces have 1 cylinder, while the second one has 2 cylinders. For each cylinder its bottom and top boundary components are colored in dark and light gray respectively and are labeled with the label of this cylinder.\\
        \emph{Second row}: corresponding ribbon graphs $G_S$ with unique vertex and labeled boundary components of two colors. The first and the second ones have genus 0, while the third one has genus 1.\\ 
        \emph{Third row}: corresponding dual ribbon graphs $G_S^*$ with unique boundary component and labeled vertices of two colors. The first two graphs, being of genus 0 and with unique boundary component, are plane trees.}
        \label{fig:examples}
    \end{figure}
    
    \subsection{Properties of counting functions}
    \label{subsec:local_polys}
    
    In this section we list some relevant properties of the counting functions $\mathcal{P}^g_{k,l}$ defined in (\ref{eq:counting_function}). We start with some notation that we will use throughout the rest of the paper. For $k,l\geq 1$ fixed:
    \begin{itemize}
        \item let $H_{k,l}$ be the hyperplane in the space $\mathbb{R}^k \times \mathbb{R}^l$ of parameters $(L,L')$ given by 
        \begin{equation}
            \label{eq:double-count}
            L_1+\ldots+L_k=L'_1+\ldots+L'_l;
        \end{equation}
        \item let $H^+_{k,l}$ be the cone $H_{k,l} \cap (\mathbb{R}_{>0}^k \times \mathbb{R}_{>0}^l)$;
        \item let $\mathcal{W}_{k,l}$ denote the set of hyperplanes of $H_{k,l}$ (the \emph{walls}) of the form 
        \[\sum_{i\in I} L_i = \sum_{j\in J} L'_j,\] 
        where $I \subset \{1,\dots,k\}, J \subset \{1,\dots,l\}$, $(I,J) \neq (\varnothing,\varnothing)$, $(I^c,J^c) \neq (\varnothing,\varnothing)$;
        \item let $\overline{\mathcal{W}_{k,l}}$ be the set of linear subspaces of $H_{k,l}$ which are intersections of several hyperplanes from $\mathcal{W}_{k,l}$ (empty intersection corresponds to $H_{k,l}$ itself);
        \item for any $W \in \overline{\mathcal{W}_{k,l}}$ let $W^\circ=W-\bigcup_{V \in \overline{\mathcal{W}_{k,l}}, V \subsetneq W} V$, i.e. $W$ minus the subspaces from $\overline{\mathcal{W}_{k,l}}$ of smaller dimension included in $W$. Note that the sets $W^\circ$ for $W \in \overline{\mathcal{W}_{k,l}}$ form a partition of $H_{k,l}$.
    \end{itemize}

    \begin{prop}
    \label{prop:counting_function_properties}
    For every $W \in \overline{\mathcal{W}_{k,l}}$ and every connected component $C$ of $H^+_{k,l} \cap W^\circ$, the function $\mathcal{P}^g_{k,l}(L,L')$ is given by a polynomial in $L,L'$ of degree at most $2g$ for $(L,L') \in C \cap (\mathbb{Z}^k \times \mathbb{Z}^l)$.
    \end{prop}
    Proposition~\ref{prop:counting_function_properties} is proven in Section~\ref{subsec:basic_properties}.
    
    In general, for a subspace $W \in \overline{\mathcal{W}_{k,l}}$, $\mathcal{P}^g_{k,l}(L,L')$ is given by different polynomials on different connected components of $H^+_{k,l} \cap W^\circ$. However, it turns out that their top-degree terms coincide.
    
    \begin{theorem}
    \label{thm:top_term}
    For all $g \geq 0, k,l \geq 1$ and every $W \in \overline{\mathcal{W}_{k,l}}$ there exists a homogeneous polynomial $P^g_W$ in the variables $L,L'$ of degree $2g$ (or identically zero) such that for all $(L,L') \in \mathbb{Z}^k \times \mathbb{Z}^l$ belonging to $H^+_{k,l} \cap W^\circ$ we have
    \[\mathcal{P}^g_{k,l}(L,L') = P^g_W(L,L') + \text{terms of degree at most } 2g-1.\]

    By point 3 of Proposition~\ref{prop:counting_function_properties}, the ``terms of degree at most $2g-1$'' are given by a polynomial of degree at most $2g-1$ on each connected component of $H^+_{k,l} \cap W^\circ$.
    \end{theorem}
    
    Note that by formula (\ref{eq:contribution_via_counting}), we are only interested in the asymptotics of (\ref{eq:k-cyl_contribution}), so (as we will see later in Section~\ref{subsec:proof_main_theorem}) it is indeed enough to study only the term of $\mathcal{P}^g_{k,l}$ of degree $2g$.
    
    Theorem~\ref{thm:top_term} is proven in sections \ref{subsec:proofThmTrees} and \ref{subsec:proofThmPositiveGenus}.
    
    In view of Proposition~\ref{prop:k-cyl_contribution}, an explicit formula is desirable for the polynomial $P^g_{V_n}$ with $V_n=\{L_1=L'_1, \ldots, L_n=L'_n\} \in \overline{\mathcal{W}_{n,n}}$. It is given in the following theorem.
    
    \begin{theorem}
    \label{thm:polys}
    For all $g \geq 0, n \geq 1$ and for all $L \in \mathbb{Z}^n$ such that $(L,L) \in H^+_{n,n} \cap V^\circ_n$, we have
    \[P^g_{V_n}(L,L) = 2^n \cdot \sum_{\substack{s_1+\ldots+s_n=g+n\\ s_i\geq 1}} p_{2s_1,\ldots,2s_n} \frac{L_1^{2s_1-2}}{(2s_1)!}\cdots \frac{L_n^{2s_n-2}}{(2s_n)!},\] 
    where the numbers $p_{2s_1,\ldots,2s_n} \in \mathbb{Z}_{>0}$ are part of a bigger collection of numbers $p_{s_1,\ldots,s_n} \in \mathbb{Z}_{>0}$ ($n\geq 1, s_i\geq 2$) that are symmetric in the indices $s_i$, and whose generating function in an infinite number of variables
    \[\mathcal{T}(t,t_2,t_3,\ldots) = 1 + \sum_{s, n \geq 1} (s-1)t^s \frac{1}{n!} \sum_{\substack{s_1+\ldots+s_n=s\\ s_i \geq 2}} p_{s_1,\ldots,s_n} t_{s_1} \cdots t_{s_n}\]
    satisfies the following relation for all $k\geq 0$: \begin{equation}
        \label{eq:relation_multivariate}
        \frac{1}{k!}[t^k]\mathcal{T}(t,t_2,t_3,\ldots)^k = [t^k] \exp\left(\sum_{i\geq 2} t_i t^i\right).
    \end{equation}
    \end{theorem}
    
    Using Lagrange inversion, (\ref{eq:relation_multivariate}) can be rewritten equivalently as 
    \begin{equation*}
        \mathcal{T}(t,t_2,\ldots)= \frac{t}{Q^{-1}(t,t_2,\ldots)}, \ Q(t,t_2,\ldots) = t\cdot \exp \left( \sum_{k=1}^{\infty} (k-1)! b_k(t_2,\ldots) t^k \right),
    \end{equation*}
    where $b_k(t_2,\ldots) = [t^k] \exp \left(\sum_{i\geq 2} t_i t^i \right)$ and functional inversion is with respect to the variable $t$. In particular, the numbers $p_{s_1,\ldots,s_n}$ can be effectively computed. We present in Table \ref{tab:p-numbers} the values of some $p_{2s_1,\ldots,2s_n}$ with small indices.
    
    \begin{table}[h]
        \centering
        \begin{tabular}{|c|c|c|c|c|c|c|c|c|c|c|}
            \hline
            $p_{2}$ & $p_{4}$ & $p_{2,2}$ & $p_{6}$ & $p_{4,2}$ & $p_{2,2,2}$ & $p_{8}$ & $p_{6,2}$ & $p_{4,4}$ & $p_{4,2,2}$ & $p_{2,2,2,2}$\\
            \hline
            1 & 2 & 1 & 24 & 18 & 11 & 720 & 600 & 684 & 486 & 335 \\
            \hline
        \end{tabular}
        \caption{Values of $p_{2s_1,\ldots,2s_n}$ with $s_1+\ldots+s_n \leq 4$.}
        \label{tab:p-numbers}
    \end{table}
    
    Theorem~\ref{thm:polys} is proven in Section~\ref{subsec:proofRecursion}. Together with Proposition~\ref{prop:k-cyl_contribution} and formula (\ref{eq:contribution_via_counting}) it implies the main Theorem~\ref{thm:main_theorem}, as well as the following explicit formulas for the volume contributions in terms of the numbers $p_{s_1,\ldots,s_n}$ (both are proven in Section~\ref{subsec:proof_main_theorem}).
    
    \begin{prop}
    \label{prop:explicit_formulas_contributions}
    For $g\geq 1$ and $1 \leq n \leq g$:
        \[\operatorname{Vol}_n(2g-2) = \frac{2}{(2g-1)!} \cdot \frac{1}{n!} \cdot \sum_{\substack{s_1+\ldots+s_n=g\\ s_i\geq 1}} p_{2s_1,\ldots,2s_n} \frac{\zeta(2s_1)}{s_1}\cdots \frac{\zeta(2s_n)}{s_n},\]
        where $\zeta$ is the Riemann zeta function.
        
    Equivalently,
        \[a_{g,n}= \frac{1}{n!} \cdot \sum_{\substack{s_1+\ldots+s_n=g\\ s_i\geq 1}} p_{2s_1,\ldots,2s_n} \cdot \prod_{i=1}^n \frac{(-1)^{s_i+1}B_{2s_i}}{2s_i\cdot (2s_i)!},\]
        where $B_i$ is the $i$-th Bernoulli number.
    \end{prop}
    
    \subsection{Analogy with Kontsevich polynomials}
    \label{subsec:analogy}
    
    As already noted in Section~\ref{subsec:remarksMainTheorem}, there is a result analogous to our Theorem~\ref{thm:top_term} for the counting functions of another family of metric ribbon graphs. More precisely, consider the counting function $\mathcal{N}_{g,n}(L_1,\ldots,L_n)$ for trivalent integral metric ribbon graphs of fixed genus $g$ and $n$ labeled boundary components of given integral perimeters $L_1,\ldots,L_n$. The top-degree terms of $\mathcal{N}_{g,n}$ are also polynomials (\emph{Kontsevich polynomials}), a result appearing (in a different form) as part of Kontsevich's proof \cite{Kon} of Witten's conjecture \cite{Wi} and also (in this form) in the paper of Norbury \cite{Nb}.
    
    Note, however, that the lower-order terms of $\mathcal{N}_{g,n}$ are global (quasi-) polynomials (\cite{Nb}), while the lower-order terms of our counting functions $\mathcal{P}^g_{k,l}$ are only piecewise polynomial.
    
    The coefficients of Kontsevich polynomials have an \emph{algebraic} interpretation as (normalized) intersection numbers of psi-classes on the Deligne-Mumford compactification of the moduli space of complex curves of genus $g$ with $n$ distinct labeled marked points. Their generating series satisfies the equations of the KdV hierarchy.
    
    The coefficients of our polynomials $P^g_W$, on the contrary, have a \emph{combinatorial} interpretation as cardinalities of certain families of metric plane trees (see the proof of Theorem~\ref{thm:polys}). This rises several open problems: on the one hand, finding the algebraic (intersection-theoretic) interpretation of the numbers $p_{s_1,\ldots,s_n}$; on the other hand, finding the combinatorial interpretation of the intersection numbers of psi-classes on the moduli space of curves. 

\section{Polynomial behavior of counting functions}
\label{sec:polyBehavior}

    \subsection{Alternative definition for counting functions}
    \label{subsec:alternativeDef}
    
    For $g \geq 0$, $k,l \geq 1$, $L \in \mathbb{Z}^k$ and $L' \in \mathbb{Z}^l$, recall from Section~\ref{subsec:metric_ribbon_graphs} the definitions of sets $\mathcal{E}_{g,k,l}$ and the counting functions $\mathcal{P}^g_{k,l}(L;L')$. Denote by $\mathcal{E}^*_{g,k,l}$ the set of isomorphism classes of genus $g$ bipartite ribbon graphs with $k$ black vertices labeled from $1$ to $k$, $l$ white vertices labeled from $1$ to $l$, and one boundary component. Note that in the special case $g=0$, the elements of $\mathcal{E}^*_{0,k,l}$ are the bipartite plane trees with $k$ black and $l$ white labeled vertices. We will refer to them simply as trees.
    
    There is a bijective correspondence between $\mathcal{E}_{g,k,l}$ and $\mathcal{E}^*_{g,k,l}$ given by passing to the \emph{dual ribbon graph}. For $G \in \mathcal{E}_{g,k,l}$ one constructs its dual ribbon graph $G^* \in \mathcal{E}^*_{g,k,l}$ as follows (see the third row of Figure \ref{fig:examples}). Vertices of $G^*$ correspond to the boundary components of $G$. For each edge $e$ in $G$ there is a corresponding edge $e^*$ in $G^*$ which joins the vertices corresponding to the boundary components on either side of $e$. The cyclic ordering of edges around a vertex in $G^*$ is inherited from the cyclic ordering of the corresponding edges along the corresponding boundary component in $G$. The vertices of $G^*$ inherit the coloring and the labels from the boundary components of $G$. The coloring condition for $G$ implies that $G^*$ is indeed bipartite. One can also check that $G^*$ has a unique boundary component (corresponding to the unique vertex of $G$) and that its genus is still $g$. Note that duality also induces a bijection between $\operatorname{Aut}(G)$ and $\operatorname{Aut}(G^*)$. In particular, $|\operatorname{Aut}(G)| = |\operatorname{Aut}(G^*)|$.
    
    Integral metrics on $G$ that give boundary components the perimeters $L$ and $L'$ are in bijective correspondence with integral metrics on $G^*$ such that the sums of edge lengths around the corresponding vertices are $L$ and $L'$ (give an edge $e^*$ of $G^*$ the length equal to the length of $e$ in $G$). By analogy, we call the numbers $L_i$ and $L'_i$ the \emph{perimeters of the vertices} of $G^*$.
    
    We can now give an alternative definition of $\mathcal{P}^g_{k,l}$, similar to the original one given in (\ref{eq:counting_function}):
    
    \begin{equation}
        \label{eq:counting_function_dual}
        \mathcal{P}^g_{k,l}(L,L') = \sum_{G^* \in \mathcal{E}^*_{g,k,l}} \frac{1}{|\operatorname{Aut}(G^*)|} \cdot \mathcal{P}_{G^*}(L,L'),
    \end{equation}
    where we denote by $\mathcal{P}_{G^*}(L,L')$ the number of integral metrics on $G^*$ giving its vertices the perimeters $L$ and $L'$. 
    
    In the following sections we will exclusively use this alternative definition of $\mathcal{P}^g_{k,l}$.
    
    \subsection{Piecewise polynomiality of counting functions}
    \label{subsec:basic_properties}
    This section is devoted to the proof of Proposition~\ref{prop:counting_function_properties} about the piecewise polynomiality of $\mathcal{P}^g_{k,l}$. The proof is based on elementary observations about the spaces of weight functions on ribbon graphs (defined below) and a result from the theory of enumeration of integer points in polyhedra.
    
    Recall the notations from Section~\ref{subsec:local_polys}. Denote by $\mathcal{L}_{k,l}$ the set of linear functions on $H_{k,l}$ of the form $\sum_{i\in I} L_i - \sum_{j\in J} L'_j$, where $I \subset \{1,\dots,k\}, J \subset \{1,\dots,l\}$, $(I,J) \neq (\varnothing,\varnothing)$, $(I^c,J^c) \neq (\varnothing,\varnothing)$. Note that the hyperplanes in $\mathcal{W}_{k,l}$ are exactly the kernels of functions from $\mathcal{L}_{k,l}$.
    
    For each $G \in \mathcal{E}^*_{g,k,l}$, let $E(G)$ be the set of edges of $G$. We call a \emph{weight function} on $G$ a function $w: E(G) \rightarrow \mathbb{R}$. The space $\mathbb{R}^{E(G)}$ of all possible weight functions on $G$ is naturally a vector space. We denote $w(e)=w_e$ and call it the \emph{weight} of the edge $e \in E(G)$. A weight function is \emph{non-negative} (\emph{positive}, \emph{integral}) if all the weights $w_e, e \in E(G)$ are non-negative (positive, integral respectively). Note that positive weight functions on $G$ are exactly the metrics on $G$ as defined in Section~\ref{subsec:metric_ribbon_graphs}. Vertex perimeters for a weight function on $G$ are defined similarly to the case of metrics on $G$. We denote by $\operatorname{vp}_G: \mathbb{R}^{E(G)} \rightarrow \mathbb{R}^k \times \mathbb{R}^l$ the linear map that sends a weight function to the perimeters it gives to the corresponding vertices. For $e \in E(G)$ let $v_b(e)$, $v_w(e)$ be the black and the white extremity of $e$ respectively. Let also $B(G) \subset E(G)$ be the set of bridges of $G$, i.e. edges whose deletion disconnects $G$.
    
    We start with some elementary observations.
    
    \begin{lemma}
    \label{lem:double-count}
    Let $G \in \mathcal{E}^*_{g,k,l}$ be a ribbon graph. Then $\operatorname{Im}(\operatorname{vp}_G) \subset H_{k,l}$. In other terms, if $w$ is a weight function on $G$ with $\operatorname{vp}_G(w) = (L,L')$, then $L_1+\ldots+L_k=L'_1+\ldots+L'_l$.
    \end{lemma}
    \begin{proof}
    Every edge contributes its weight to the perimeters of both its black and white extremities. Hence both the sum of the perimeters of black vertices of $G$ and the sum of the perimeters of its white vertices are equal to $\sum_{e \in E(G)} w_e$.
    \end{proof}
    
    \begin{lemma}[Bridge weight]
    \label{lem:bridge_length_formula}
    Let $G \in \mathcal{E}^*_{g,k,l}$ be a ribbon graph and let $w$ be a weight function on $G$ with $\operatorname{vp}_G(w) = (L,L')$. Let also $e\in B(G)$ be a bridge. Then 
    \begin{equation}
    \label{eq:bridge_length_formula}
        w_e = \sum_{i \in I} L_i - \sum_{j \in J} L'_j = \sum_{j \in J^c} L'_j - \sum_{i \in I^c} L_i,
    \end{equation} 
    where $I \subset \{1,\dots,k\}$ and $J \subset \{1,\dots,l\}$ are the labels of black and white vertices in the connected component of $G-e$ containing $v_b(e)$.
    \end{lemma}
    \begin{proof}
    Denote by $(G-e)_b$ the connected component of $G-e$ containing $v_b(e)$. Again we note that every edge contributes its weight to the perimeters of both its black and white extremities. Hence the sum of edge weights in $(G-e)_b$ is $\sum_{j \in J} L'_j$ (total contribution to white vertices of $(G-e)_b$) and at the same time $\sum_{i \in I} L_i - w_e$ (total contribution to black vertices of $(G-e)_b$). The first equality follows. The second equality follows from Lemma~\ref{lem:double-count}
    \end{proof}
    
    For a bridge $e \in B(G)$ we denote by $f_e(L,L')$ the linear function in $\mathcal{L}_{k,l}$ given by (\ref{eq:bridge_length_formula}). We do not specify the dependency on $G$ in the notation $f_e$ as it will always be clear from the context.
    
    \begin{lemma}[Weight functions on trees]
    \label{lem:metrics_on_trees}
    Let $G \in \mathcal{E}^*_{0,k,l}$ be a tree. Then the linear map $\operatorname{vp}_G: \mathbb{R}^{E(G)} \rightarrow \mathbb{R}^k \times \mathbb{R}^l$ induces an isomorphism between $\mathbb{R}^{E(G)}$ and $H_{k,l}$. Moreover, $w \in \mathbb{R}^{E(G)}$ is integral if and only if $\operatorname{vp}_G(w) \in \mathbb{Z}^k\times \mathbb{Z}^l$.
    \end{lemma}
    \begin{proof}
    We first show the $\operatorname{vp}_G$ is injective. Suppose $w$ is such that $\operatorname{vp}_G(w)=(L,L')$. Since in a tree all edges are bridges, by Lemma~\ref{lem:bridge_length_formula} $w_e = f_e(L,L')$ for every $e \in E(G)$. Hence such $w$ is necessarily unique. We now show that $\operatorname{Im}(\operatorname{vp}_G) = H_{k,l}$. By Lemma~\ref{lem:double-count} $\operatorname{Im}(\operatorname{vp}_G) \subset H_{k,l}$. In the other direction, let $(L,L') \in H_{k,l}$. Construct a weight function $w$ by setting $w_e = f_e(L,L')$ for every $e \in E(G)$. The proof that $\operatorname{vp}_G(w) = (L,L')$ is then a simple computation.
    
    The integrality of $w$ implies $\operatorname{vp}_G(w) \in \mathbb{Z}^k\times \mathbb{Z}^l$ by definition of $\operatorname{vp}_G$. Opposite implication holds because linear forms in $\mathcal{L}_{k,l}$ have integer coefficients.
    \end{proof}
    
    \begin{lemma}
    Let $G \in \mathcal{E}^*_{g,k,l}$ be a ribbon graph. Then $\operatorname{Im}(\operatorname{vp}_G) = H_{k,l}$. Moreover, for every $(L,L') \in H_{k,l}$, $\operatorname{vp}_G^{-1}(L,L')$ is an affine subspace of $\mathbb{R}^{E(G)}$ of dimension $2g$.
    \end{lemma}
    \begin{proof}
    Again, $\operatorname{Im}(\operatorname{vp}_G) \subset H_{k,l}$ by Lemma~\ref{lem:double-count}. To prove the opposite inclusion, fix $(L,L') \in H_{k,l}$. Choose a spanning tree $T$ of $G$. Set $w_e=0$ for $e \in E(G) \setminus E(T)$. By Lemma~\ref{lem:metrics_on_trees} there exists a weight function $w' \in \mathbb{R}^{E(T)}$ such that $\operatorname{vp}_T(w')= (L,L')$. Set $w_e = w'_e$ for $e \in E(T)$. Thus constructed weight function $w \in \mathbb{R}^{E(G)}$ satisfies $\operatorname{vp}_G(w)=(L,L')$.
    
    The dimension of $\operatorname{vp}_G^{-1}(L,L')$ is $\dim \mathbb{R}^{E(G)} - \dim H_{k,l} = |E(G)| - (k+l-1) = 2g$ by Euler's formula applied to the ribbon graph $G$.
    \end{proof}
    
    By definition, for $(L,L') \in \mathbb{Z}^k \times \mathbb{Z}^l$, $\mathcal{P}_G(L,L')$ is the number of \emph{integer} solutions $w = \{w_e\}_{e \in E(G)}$ to the following system:
    \begin{equation}
    \label{eq:one_graph_system}
    \begin{cases}
    w \in \operatorname{vp}_G^{-1}(L,L') \\
    w_e > 0,\ e\in E(G).
    \end{cases}
    \end{equation}
    
    \begin{lemma}
    \label{lem:linear_functions_on_subspace}
    Let $G \in \mathcal{E}^*_{g,k,l}$ be a ribbon graph and let $(L,L') \in H_{k,l}$. Regard the coordinates $w_e$ as linear functions on the affine subspace $\operatorname{vp}_G^{-1}(L,L')$. Then for all $e \in B(G)$, $w_e$ is constant with value $f_e(L,L')$. All other functions $w_e, e \in E(G) \setminus B(G)$ are non-constant.
    \end{lemma}
    \begin{proof}
    The first claim follows from Lemma~\ref{lem:bridge_length_formula}. Let now $e \in E(G)\setminus B(G)$. Since $e$ is not a bridge, there exists a cycle in $G$ containing $e$. Since $G$ is bipartite, this cycle has even length. One can now change the value of $w_e$ while staying inside $\operatorname{vp}_G^{-1}(L,L')$ by alternately adding and subtracting some $t \in \mathbb{R}$ to/from the weights of consecutive edges of this cycle. So $w_e$ is indeed non-constant on $\operatorname{vp}_G^{-1}(L,L')$.
    \end{proof}
    
    In what follows we will use some terminology coming from the polyhedron theory. We refer the reader to \cite{Barvinok} for details.
    
    Lemma~\ref{lem:linear_functions_on_subspace} allows to define in $\operatorname{vp}_G^{-1}(L,L')$ the following polytope (i.e. a bounded polyhedron):
    \begin{equation}
        M_G(L,L') = \{w \in \operatorname{vp}_G^{-1}(L,L') : w_e \ge 0, e \in E(G) \setminus B(G)\}.
    \end{equation} 
    Then it follows from (\ref{eq:one_graph_system}) that for $(L,L') \in H_{k,l} \cap (\mathbb{Z}^k \times \mathbb{Z}^l)$:
    \begin{equation}
        \label{eq:one_graph_poly}
        \mathcal{P}_G(L,L') = \left( \prod_{e \in B(G)} \mathbf{1}_{f_e(L,L')>0} \right) \cdot |\operatorname{int} M_G(L,L') \cap \mathbb{Z}^{E(G)}|,
    \end{equation}
    where $\mathbf{1}$ denotes the indicator function and $\operatorname{int}$ denotes the interior relative to $\operatorname{vp}_G^{-1}(L,L')$.
    
    Recall that for a polyhedron $P$ and a point $p$ in $\mathbb{R}^d$ the \emph{cone of feasible directions to $P$ at $p$} is defined as 
    \[\operatorname{fcone}(P,p) = \{v \in \mathbb{R}^d: p+\varepsilon v \in P \text{ for some } \varepsilon>0\}.\]
    For example, $p$ is an interior point of $P$ if and only if $\operatorname{fcone}(P,p) = \mathbb{R}^d$.
    
    \begin{lemma}
    \label{lem:vertex_characterisation}
    Let $G \in \mathcal{E}^*_{g,k,l}$ be a ribbon graph and let $(L,L') \in H_{k,l}$. Then $w \in \operatorname{vp}_G^{-1}(L,L')$ is a vertex of $M_G(L,L')$ if and only if $w_e \ge 0$ for all $e \in E(G)\setminus B(G)$ and the edges in $F = \{e \in E(G)\setminus B(G) : w_e > 0\}$ form a forest.
    
    The cone of feasible directions to $M_G(L,L')$ at such vertex $w$ is given by the system
        \[
        \begin{cases}
        v \in \operatorname{vp}_G^{-1}(0,0),\\
        v_e \ge 0,\ e \in E(G) \setminus (B(G) \cup F).
        \end{cases}
        \]
    In particular, it only depends on $F$ and not on $L,L'$.
    \end{lemma}
    \begin{proof}
    
    Suppose $w$ is such that $w_e \ge 0$ for all $e \in E(G)\setminus B(G)$ and the edges in $F = \{e \in E(G)\setminus B(G) : w_e > 0\}$ form a forest. The first condition ensures that $w \in M_G(L,L')$, by definition of $M_G(L,L')$. The second condition ensures that $w$ is not a midpoint of a segment (not reduced to a point) whose endpoints lie in $M_G(L,L')$. Indeed, suppose $w=(w' + w'')/2$ with $w', w'' \in M_G(L,L')$. Since $w_e=0, w'_e \ge 0, w''_e \ge 0$ for $e \in E(G) \setminus (B(G) \cup F)$, necessarily $w'_e = w''_e = 0$ for $e \in E(G) \setminus (B(G) \cup F)$. But then, by Lemma~\ref{lem:metrics_on_trees}, the weights $w'_e, w''_e, e \in B(G) \cup F$ of $w'$ and $w''$ are uniquely determined ($B(G) \cup F$ forms a collection of trees) and are equal to the corresponding weights of $w$. Hence $w'=w''=w$ and $w$ is an extreme point of $M_G(L,L')$, hence a vertex.
    
    Conversely, let $w$ be a vertex of $M_G(L,L')$. Then $w_e \geq 0$ for all $e \in E(G)\setminus B(G)$ because $w \in M_G(L,L')$. Vertices of $M_G(L,L')$ are exactly its extreme points, but if there were a cycle of edges of positive weight in $w$, one would be able to modify the weights in this cycle by alternately adding and subtracting $\varepsilon$ or $-\varepsilon$ to/from the weights of consecutive edges of this cycle, for some small $\varepsilon>0$, and write $w$ as a midpoint of these two modifications, which still belong to $M_G(L,L')$. So $F$ indeed forms a forest.
    
    The second claim follows from the definition of the cone of feasible directions and the defining system for $M_G(L,L')$ (the weights of edges in $E(G) \setminus (B(G) \cup F)$ can only be perturbed in the positive direction, while the weights of edges in $F$ can be perturbed arbitrarily).
    \end{proof}
    
    For a vertex $w$ of $M_G(L,L')$ we call the set $F \subset E(G) \setminus B(G)$ as in Lemma~\ref{lem:vertex_characterisation} the \emph{support} of $w$.
    
    \begin{lemma}
    \label{lem:polytope_dependence}
    Fix a ribbon graph $G \in \mathcal{E}^*_{g,k,l}$, a subspace $W \in \overline{\mathcal{W}_{k,l}}$ and a connected component $C$ of $H^+_{k,l} \cap W^\circ$.
    
    There exist subsets $F_1,\ldots, F_n \subset E(G) \setminus B(G)$ each forming a forest, such that each polytope $M_G(L,L')$ with $(L,L') \in C$ has $n$ vertices $v_1(L,L'), \ldots, v_n(L,L')$ with supports $F_1,\ldots, F_n$ respectively. For each $i$ the coordinates of $v_i(L,L')$ are either identically zero or are linear functions (of $L,L'$) from $\mathcal{L}_{k,l}$. If $(L,L')  \in \mathbb{Z}^k\times \mathbb{Z}^l$, then all of the vertices of $M_G(L,L')$ are integral. 
    
    Moreover, for each $i$ the cone of feasible directions $\operatorname{fcone}(M_G(L,L'), v_i(L,L'))$ is constant (does not depend on $L,L'$).
    \end{lemma}
    \begin{proof}
    Consider a subset $F \subset E(G)\setminus B(G)$ which forms a forest. Then $B(G) \cup F$ also forms a forest, i.e. a collection of trees. So by Lemma~\ref{lem:metrics_on_trees} there exists a (unique) weight function $w$ on $G$ such that $w_e=0$ for $e \notin B(G) \cup F$ if and only if for each constituent tree of $B(G) \cup F$ the sums of perimeters of its black and its white vertices are equal (condition 1). If such $w$ exists, then by Lemma~\ref{lem:vertex_characterisation} $F$ is a support of a vertex of $M_G(L,L')$ if and only if $w_e>0$ for $e \in F$ (condition 2).
    
    Note that conditions 1 and 2 are equivalent to the fact that several linear functions from $\mathcal{L}_{k,l}$ are zero (for condition 1) or positive (for condition 2, because edge weights for trees are given by linear functions from $\mathcal{L}_{k,l}$ by Lemma~\ref{lem:bridge_length_formula}).
    
    By definition of the family of subspaces $\overline{\mathcal{W}_{k,l}}$, when $(L,L')$ stays in $C$, the signs ($+$, $-$ or $0$) of all linear forms in $\mathcal{L}_{k,l}$ remain constant (we do not leave or enter any new walls from $\mathcal{W}_{k,l}$). So for each $F$ conditions 1 and 2 are either satisfied everywhere or nowhere on $C$. Hence each $F$ is a support of a (unique) vertex of $M_G(L,L')$ either for all $(L,L') \in C$ or for none of them. This proves the first claim. 
    
    It follows from the discussion above that the edge weights at the vertices of $M_G(L,L')$ are either identically zero or are linear functions from $\mathcal{L}_{k,l}$. Integrality claim follows since linear functions from $\mathcal{L}_{k,l}$ have integer coefficients. Finally, the cone of feasible directions at a vertex is determined by its support by Lemma~\ref{lem:vertex_characterisation}.
    \end{proof}
    
    In the proof of Proposition~\ref{prop:counting_function_properties} we will also need the following result from the theory of enumeration of integer points in polyhedra (see \cite{Barvinok} for details).
    
    \begin{theorem}[Theorem 18.1 in \cite{Barvinok}]
    \label{thm:barvinok}
    Let $\{P_{\alpha} : \alpha \in A\}$ be a family of $d$-dimensional polytopes in $\mathbb{R}^d$ with vertices $v_1(\alpha),\ldots, v_n(\alpha)$ such that $v_i(\alpha) \in \mathbb{Z}^d$ and the cones of feasible directions at $v_i(\alpha)$ do not depend on $\alpha$:
    \[\operatorname{fcone}(P_{\alpha},v_i(\alpha)) = \operatorname{const}_i,\ i=1,\ldots, n. \]
    Then there exists a polynomial $p: (\mathbb{R}^d)^n \rightarrow \mathbb{R}$ such that 
    \[|\operatorname{int} P_{\alpha} \cap \mathbb{Z}^d| = p(v_1(\alpha),\ldots, v_n(\alpha)).\]
    \end{theorem}
    
    \begin{remark}
    \label{rmk:barvinok_remark}
    In Theorem~\ref{thm:barvinok} the polynomial $p$ is of degree $d$. Its top-degree term $p_{top}$ gives the volume of $P_{\alpha}$ with respect to the Lebesgue measure on $\mathbb{R}^d$ normalized so that the covolume of the lattice $\mathbb{Z}^d$ is equal to 1. 
    
    Indeed, when $c\rightarrow \infty$, 
    \[|\operatorname{int} (c\cdot P_{\alpha}) \cap \mathbb{Z}^d| \sim c^d \cdot \operatorname{Vol}(P_{\alpha}),\] 
    while it is the same as 
    \[p(c \cdot v_1(\alpha), \ldots, c \cdot v_n(\alpha)) \sim c^{\operatorname{deg}(p)} \cdot p_{top}(v_1(\alpha), \ldots, c \cdot v_n(\alpha)).\]
    \end{remark}
    
    We now pass to the proof of Proposition~\ref{prop:counting_function_properties}.
    
    \begin{proof}[Proof of Proposition~\ref{prop:counting_function_properties}]
    
    Since $\mathcal{P}^g_{k,l}(L,L')$ is a weighted sum of $\mathcal{P}_{G}(L,L')$ over all $G$, it is enough to show that for every $G \in \mathcal{E}^*_{g,k,l}$, every $W \in \overline{\mathcal{W}_{k,l}}$ and every connected component $C$ of $H^+_{k,l} \cap W^\circ$, the function $\mathcal{P}_{G}(L,L')$ is given by a polynomial in $L,L'$ of degree at most $2g$ for $(L,L') \in C \cap (\mathbb{Z}^k \times \mathbb{Z}^l)$.
    
    Fix $G, W, C$ as above and recall the formula (\ref{eq:one_graph_poly}). As in the proof of Lemma~\ref{lem:polytope_dependence}, when $(L,L')$ stays in $C$, the signs ($+$, $-$ or $0$) of all linear forms in $\mathcal{L}_{k,l}$ remain constant. It means that the product of indicator functions in (\ref{eq:one_graph_poly}) is constant on $C$. 
    
    By Lemma~\ref{lem:polytope_dependence}, for $(L,L') \in C \cap (\mathbb{Z}^k \times \mathbb{Z}^l)$ the polytopes $M_G(L,L')$ have a fixed number of integral vertices with equal cones of feasible directions at the corresponding vertices. In particular, they all have the same dimension. If their common dimension is less then the dimension of $\operatorname{vp}^{-1}_G(L,L')$, their interiors are empty and the second term of (\ref{eq:one_graph_poly}) is identically zero. Otherwise, by Theorem~\ref{thm:barvinok} the second term is a polynomial of degree $2g$ in the coordinates of the vertices, which are themselves linear functions of $L,L'$ (by Lemma~\ref{lem:polytope_dependence}), which proves Proposition~\ref{prop:counting_function_properties}.
    
    Note that, formally, one cannot apply Theorem~\ref{thm:barvinok} to a family of polytopes $M_G(L,L')$ belonging to different parallel affine subspaces $\operatorname{vp}_G^{-1}(L,L')$. However, one can first identify each $\operatorname{vp}_G^{-1}(L,L')$ with $\operatorname{vp}_G^{-1}(0,0)$ by a translation. The translation vector (depending linearly on $L,L'$) can be chosen as the unique weight function $w$ on $G$ such that $w_e=0$ for $e \notin E(T)$ for some fixed spanning tree $T$ of $G$. For $(L,L') \in \mathbb{Z}^k \times \mathbb{Z}^l$ this vector will be integral and so the integral lattice of $\operatorname{vp}_G^{-1}(L,L')$ is identified by this translation with the integral lattice of $\operatorname{vp}_G^{-1}(0,0)$ and we can apply Theorem~\ref{thm:barvinok}.
    \end{proof}
    
    \begin{remark}
    \label{rmk:topDegreeTerm}
    It follows from the proof of Proposition~\ref{prop:counting_function_properties} and Remark~\ref{rmk:barvinok_remark} that for every $W \in \overline{\mathcal{W}_{k,l}}$ and every connected component $C$ of $H^+_{k,l} \cap W^\circ$, the top-degree term of the polynomial which gives $\mathcal{P}^g_{k,l}(L,L')$ for $(L,L') \in C \cap (\mathbb{Z}^k \times \mathbb{Z}^l)$ is equal to:
    \begin{equation}
    \label{eq:top-degree-term}
        \sum_{G \in \mathcal{E}^*_{g,k,l}} \frac{1}{|\operatorname{Aut}(G)|} \cdot \left( \prod_{e \in B(G)} \mathbf{1}_{f_e(L,L')>0} \right) \cdot \operatorname{Vol} M_G(L,L'),
    \end{equation}
    where $\operatorname{Vol}$ denotes the volume with respect to the Lebesgue measure on $\operatorname{vp}_G^{-1}(L,L')$ normalized so that the covolume of its integer lattice is equal to 1. 
    \end{remark}

    \subsection{Polynomiality of top-degree term for \texorpdfstring{$g=0$}{g=0}}
    \label{subsec:proofThmTrees}
    Having proved Proposition~\ref{prop:counting_function_properties} about the piecewise polynomiality of the counting functions, we now pass to the finer question of polynomiality of their top-degree terms (Theorem~\ref{thm:top_term}). In this section we deal with the case $g=0$. We will use the notation and the results of Section~\ref{subsec:basic_properties}.
    
    Recall that the elements of $\mathcal{E}^*_{0,k,l}$ are bipartite plane trees with $k$ black and $l$ white labeled vertices. Note that every tree $G \in \mathcal{E}^*_{0,k,l}$ has no non-trivial automorphisms, i.e. $|\operatorname{Aut}(G)|=1$. Indeed, an automorphism of $G$ must simultaneously preserve a leaf of $G$ (automorphisms preserve the labelling) and the order of edges along the unique boundary, so it is necessarily the identity.
    
    Let $G \in \mathcal{E}^*_{0,k,l}$ and $(L,L') \in H_{k,l}$. By Lemma~\ref{lem:metrics_on_trees} there exists a unique weight function $w$ on $G$ with $\operatorname{vp}_G(w) = (L,L')$. We will say that $G$ is \emph{positive at} $(L,L')$ if $w$ is positive.
    
    \begin{proof}[Proof of Theorem~\ref{thm:top_term} for $g=0$]
    Fix $W \in \overline{\mathcal{W}_{k,l}}$. By Proposition~\ref{prop:counting_function_properties} for $g=0$, the counting function $\mathcal{P}^0_{k,l}(L,L')$ is given by a polynomial of degree 0, i.e. a constant, for $(L,L') \in \mathbb{Z}^k \times \mathbb{Z}^l$ belonging to each connected component of $H^+_{k,l} \cap W^\circ$. 
    To prove Theorem~\ref{thm:top_term} for $g=0$ and the subspace $W$, we have to show that these constants for different connected components are equal.
    
    By Lemma~\ref{lem:metrics_on_trees} for every tree $G \in \mathcal{E}^*_{0,k,l}$ and every $(L,L') \in H_{k,l}$ there exists a unique weight function $w$ on $G$ with $\operatorname{vp}_G(w) = (L,L')$. This weight function $w$ is an integral metric if and only if $(L,L') \in H_{k,l} \cap (\mathbb{Z}^k \times \mathbb{Z}^l)$ and all the edge weights are positive. By Lemma~\ref{lem:bridge_length_formula}, the weight $w_e$ of every edge $e \in E(G)$ is given by some linear function $f_e$ in $L,L'$ from the set $\mathcal{L}_{k,l}$.  Hence, for $(L,L') \in H_{k,l} \cap (\mathbb{Z}^k \times \mathbb{Z}^l)$ we have
    \begin{equation}
    \label{eq:one_tree_poly}
    \mathcal{P}_G(L,L') = \prod_{e \in E(G)} \mathbf{1}_{f_e(L,L')>0}.
    \end{equation}
    Combining this with $|\operatorname{Aut}(G)| = 1$, we finally get for $(L,L') \in H_{k,l} \cap (\mathbb{Z}^k \times \mathbb{Z}^l)$:
    \begin{equation}
    \label{eq:numberOfPositiveTrees}
    \mathcal{P}^0_{k,l}(L,L') = \sum_{G \in \mathcal{E}^*_{0,k,l}} \prod_{e \in E(G)} \mathbf{1}_{f_e(L,L')>0}.
    \end{equation}
    
    Clearly, for any $(L,L') \in H_{k,l}$ the right-hand side of (\ref{eq:one_tree_poly}) is equal to 1 if and only if $G$ is positive at $(L,L')$. The right-hand side of (\ref{eq:numberOfPositiveTrees}) is then the number of trees positive at $(L,L')$ for all $(L,L') \in H_{k,l}$.
    
    By definition of the family of subspaces $\overline{\mathcal{W}_{k,l}}$, when $(L,L')$ stays in a fixed connected component of $H^+_{k,l} \cap W'$, the signs ($+$, $-$ or $0$) of all linear forms in $\mathcal{L}_{k,l}$ remain constant (we do not leave or enter any walls from $\mathcal{W}_{k,l}$). Thus the right-hand side of (\ref{eq:numberOfPositiveTrees}), i.e. the number of positive trees at $(L,L')$, is constant on each connected component of $H^+_{k,l} \cap W'$. Note that this is just a special case of Proposition~\ref{prop:counting_function_properties}.
    
    Take now $V \in \overline{\mathcal{W}_{k,l}}$ such that $V \subset W$, $V$ is of codimension 1 in $W$ and $V$ intersects the open cone $H^+_{k,l} \cap W$. It is enough to prove that the number of trees positive at $(L,L')$ does not change when the point $(L,L')$ traverses $V$ inside $W$.
    
    Consider a (small) oriented linear path $\gamma \subset H^+_{k,l} \cap W$ transversal to $V$. When $(L,L')$ reaches $V$ along $\gamma$, certain positive trees (forming a subset $\mathcal{D}^- \subset \mathcal{E}^*_{0,k,l}$) cease to be positive (some of their edges become zero-weight), and the number of positive trees decreases by $|\mathcal{D}^-|$. In turn, when $(L,L')$ continues along $\gamma$ past the point of intersection with $V$, some non-positive trees (forming a subset $\mathcal{D}^+ \subset \mathcal{E}^*_{0,k,l}$) become positive, and the number of positive trees increases by $|\mathcal{D}^+|$. It is now enough to show that $|\mathcal{D}^-|=|\mathcal{D}^+|$. We will do this by establishing a bijection between $\mathcal{D}^-$ and $\mathcal{D}^+$, which we now describe.
    
    Take a tree $T^- \in \mathcal{D}^-$. When $(L,L') = \gamma \cap V$, the tree $T^-$ has some zero-weight edges. We call these edges \emph{bad} and the other ones \emph{good}. Now \emph{flip} (in arbitrary order) all bad edges of $T^-$ as in Figure \ref{fig:zero_edge_modif_1} to get a tree $T^+$. If there are several bad edges emanating from the same vertex of $T^-$ and which are consecutive for the circular order around this vertex (Figure \ref{fig:zero_edge_modif_2}), flip them together, preserving the circular order. If one regards $T^-$ as a collection of trees with good edges which are connected by bad edges, the procedure amounts to ``rotating'' each of these trees clockwise with respect to the bad edges (see Figure \ref{fig:trees_degen}). We claim that this procedure is well-defined, that $T^+ \in \mathcal{D}^+$ and that the procedure gives a bijection between $\mathcal{D}^-$ and $\mathcal{D}^+$.
    
    \begin{figure}
        \centering
        \begin{subfigure}{0.45\textwidth}
            \setlength{\unitlength}{\textwidth}
            \begin{picture}(1,0.28)
            \put(0,0){\includegraphics[width=\textwidth]{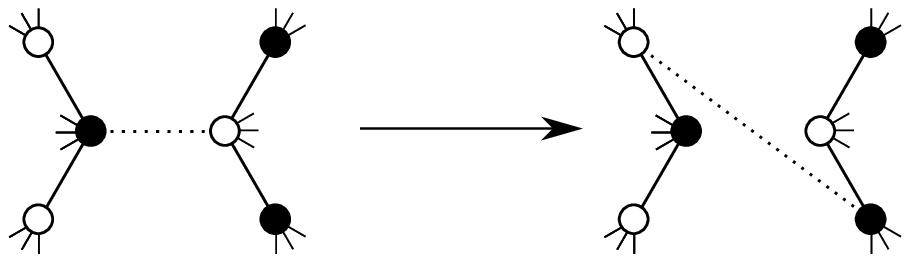}}
            \put(0.14,0.1){$e_0$}
            \put(0.01,0.18){$e_1$}
            \put(0.285,0.075){$e_2$}
            \put(0.81,0.075){$\widehat{e}_0$}
            \put(0.67,0.17){$\widehat{e}_1$}
            \put(0.95,0.075){$\widehat{e}_2$}
            \end{picture}
            \caption{}
            \label{fig:zero_edge_modif_1}
        \end{subfigure}
        \hspace{0.05\textwidth}
        \begin{subfigure}{0.45\textwidth}
            \centering
            \includegraphics[width=\textwidth]{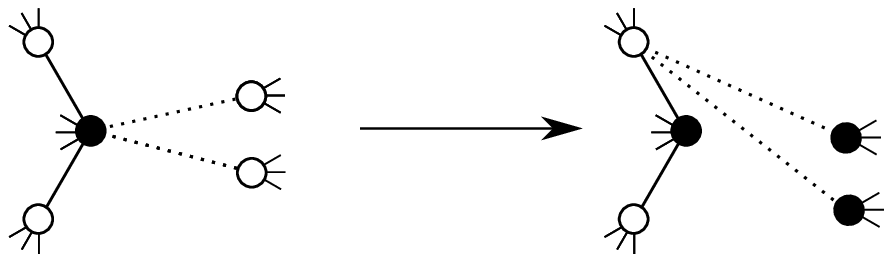}
            \caption{}
            \label{fig:zero_edge_modif_2}
        \end{subfigure}
        \caption{Flipping zero-weight edges.}
    \end{figure}
    
    The fact that the order of flips does not matter follows from the ``rotation of trees'' picture. The only obstacle to the flipping itself is when there is a vertex of $T^-$ which is only incident to bad edges, but this is not possible in our case. Indeed, when $(L,L') = \gamma \cap V$, the perimeter of this vertex must be zero (being the sum of weights of bad edges). But this implies that $\gamma \cap V \notin H^+_{k,l}$, a contradiction. So the procedure is well-defined.
    
    We now show that $T^+ \in \mathcal{D}^+$. First, consider an arbitrary tree $T$ with three of its edges $e_0, e_1, e_2$ arranged as in Figure \ref{fig:zero_edge_modif_1} and suppose we have flipped the edge $e_0$ to get a tree $\widehat{T}$. For every edge $e$ of $T$ let $\widehat{e}$ be the corresponding edge of $\widehat{T}$. As before, for an edge $e$ we denote by $f_e$ the linear function from $\mathcal{L}_{k,l}$ giving its weight as a function of vertex perimeters. Lemma~\ref{lem:bridge_length_formula} implies that $f_{\widehat{e}_0} = -f_{e_0}$, $f_{\widehat{e}_1} = f_{e_1} + f_{e_0}$, $f_{\widehat{e}_2} = f_{e_2} + f_{e_0}$ and $f_{\widehat{e}} = f_e$ for all other edges $e$ of $T$. In particular, if, restricted to $\gamma$, $f_{e_0}$ changes sign from positive to negative at $\gamma \cap V$ and $f_{e_1}$ and $f_{e_2}$ remain positive, then $f_{\widehat{e}_0}$ changes sign from negative to positive at $\gamma \cap V$ and $f_{\widehat{e}_1}$ and $f_{\widehat{e}_2}$ remain positive. If we now apply this observation to each flip we make to get from $T^-$ to $T^+$, we see that $T^+ \in \mathcal{D}^+$.
    
    Finally, this procedure gives a bijection between $\mathcal{D}^-$ and $\mathcal{D}^+$ because one can construct an inverse map from $\mathcal{D}^+$ to $\mathcal{D}^-$ by flipping in the opposite direction the bad edges of trees from $\mathcal{D}^+$.
    \end{proof}
    
    As an illustration to the proof of Theorem~\ref{thm:top_term} for $g=0$, we present in Figure \ref{fig:trees_degen} the case $k=l=3$, $W=\{L_1=L'_1, L_2=L'_2, L_3=L'_3\}$, $V=\{L_1=L_2\}$.
    
    \begin{figure}
        \centering
        \includegraphics[width=\textwidth]{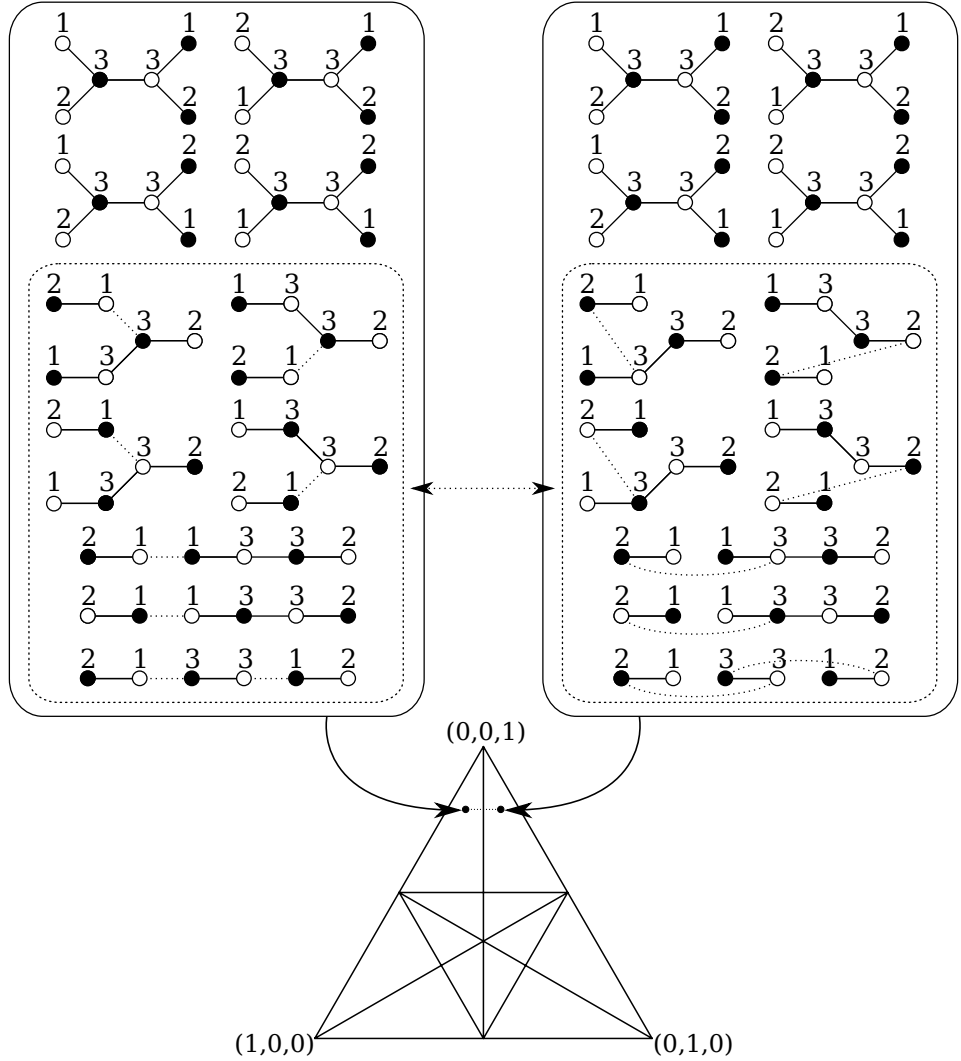}
        \caption{Illustration to the proof of Theorem~\ref{thm:top_term} for $g=0$, where $k=l=3$, $W=\{L_1=L'_1, L_2=L'_2, L_3=L'_3\}$, $V=\{L_1=L_2\}$. At the bottom of the figure we see the projectivization of the three-dimensional cone $H^+_{k,l} \cap W = \{L_1=L'_1, L_2=L'_2, L_3=L'_3, L_1>0, L_2>0, L_3>0\}$. The cone is divided by 6 codimension 1 subspaces from $\overline{\mathcal{W}_{k,l}}$ into 12 ``cells''. The number of trees positive at each point inside of each cell is the same (11 in this case). At the top of the figure are given the sets of positive trees corresponding to two particular points of the cone. Going from one of the points to the other, we must cross the subspace $V$. When we reach $V$ from one side, certain trees cease to be positive --- some of their edges become zero-weight (these edges are marked by dotted lines). In turn, when we continue to the other cell, certain trees become positive. The procedure described in the proof provides a bijection between these sets of trees. In the figure, the corresponding trees are opposite to each other.}
        \label{fig:trees_degen}
    \end{figure}

    \begin{remark}
    \label{rmk:P0CountsPositiveTrees}
    It follows from the proof of Theorem~\ref{thm:top_term} for $g=0$ that for all $k,l\ge 1$ and $W \in \overline{\mathcal{W}_{k,l}}$
    \[P^0_W(L,L') = \sum_{G \in \mathcal{E}^*_{0,k,l}} \prod_{e \in E(G)} \mathbf{1}_{f_e(L,L')>0}\]
    is constant for $(L,L') \in H^+_{k,l} \cap W^\circ$ and it is the number of trees from $\mathcal{E}^*_{0,k,l}$ positive at $(L,L')$.
    \end{remark}

    \subsection{From \texorpdfstring{$g=0$}{g=0} to  higher genera}
    \label{subsec:proofThmPositiveGenus}
    
    In this section we complete the proof of the polynomiality of the top-degree term of the counting functions $\mathcal{P}^g_{k,l}$ (Theorem~\ref{thm:top_term}) for $g>0$. In fact we prove a stronger statement (Theorem~\ref{thm:explicitPolyWall} below), of which Theorem~\ref{thm:top_term} is a direct corollary.
    
    \begin{theorem}
    \label{thm:explicitPolyWall}
    For every $g \geq 0, k,l \geq 1$ and every $W \in \overline{\mathcal{W}_{k,l}}$ there exists a family of subspaces $U_{W,b,w} \in \overline{\mathcal{W}_{k+2g_1,l+2g_2}}$ (where $b=(b_1,\ldots,b_k)$, $w=(w_1,\ldots,w_l)$ are vectors of non-negative integers and $b_1+\ldots+b_k=g_1$, $w_1+\ldots+w_l=g_2$)
    such that for every connected component $C$ of $H^+_{k,l} \cap W^\circ$, the top-degree term of the polynomial which gives $\mathcal{P}^g_{k,l}(L,L')$ for $(L,L') \in C \cap (\mathbb{Z}^k \times \mathbb{Z}^l)$ is equal to
    \begin{equation}
    \label{eq:localPolyExplicit}
        2^{-2g} \cdot \sum_{\substack{b_1+\ldots+b_k+w_1+\ldots+w_l=g\\b_i,w_i \ge 0}} P^0_{U_{W,b,w}} \cdot \prod_{i=1}^k \frac{L_i^{2b_i}}{(2b_i+1)!} \cdot \prod_{j=1}^l \frac{{L'_j}^{2w_j}}{(2w_j+1)!}.
    \end{equation}
    In particular, it does not depend on the connected component $C$.
    \end{theorem}
    
    The subspaces $U_{W,b,w}$ in Theorem~\ref{thm:explicitPolyWall} have an explicit description which can be found in the proof of the Theorem.
    
    \begin{proof}[Proof of Theorem~\ref{thm:top_term} for $g>0$]
    It is clear from Theorem~\ref{thm:explicitPolyWall} that one can take $P^g_W(L,L')$ equal to (\ref{eq:localPolyExplicit}).
    \end{proof}
    
    To prove Theorem~\ref{thm:explicitPolyWall} we will need a result from the theory of ribbon graphs, which we now introduce.
    
    In \cite{Chapuy} Chapuy introduced an operation (``slicing'') on ribbon graphs with one boundary component which decreases the genus. Roughly speaking, the idea is as follows. In planar trees (genus 0 ribbon graphs with one boundary component), when we go along the unique boundary, we visit the corners adjacent to any fixed vertex in the same order as they are situated around that vertex. However, for the higher genus ribbon graphs this is no longer true in general. Such violation of order at a vertex allows us to ``slice'' this vertex into 3 new ones, preserving all edges and decreasing the genus of the ribbon graph by 1.
    
    By iterating the slicing operation, one can obtain a genus 0 ribbon graph, i.e. a plane tree. To get back the initial ribbon graph, one has to ``glue'' some of the vertices of the tree. This suggests that there should be a bijection between ribbon graphs with one boundary component and plane trees with some decoration of their vertices. The complication is that at each step of slicing the choice of a vertex to slice is not canonical. Nevertheless, in \cite{ChapuyFerayFusy} Chapuy, F\'eray and Fusy gave such (non-explicit) bijection, which we will now present. This bijection also applies to vertex-labeled bipartite ribbon graphs ( \cite[section 3.3]{ChapuyFerayFusy}), so we state right away the version for these graphs.
    
    A bipartite ribbon graph is \emph{rooted} if it has a distinguished edge. For $g\ge 0$, $k,l \ge 1$ let $\mathcal{E}_{g,k,l}^{*,root}$ be the set of rooted bipartite ribbon graphs of genus $g$, 1 boundary component, $k$ black and $l$ white labeled vertices.
    
    Let $k,l \ge 1$. A \emph{C-decorated $(k,l)$-tree} is a triple $(T, \sigma_b, \sigma_w)$, where $T$ is a rooted bipartite plane tree with $k$ black and $l$ white \emph{non-labeled} vertices and $\sigma_b, \sigma_w$ are permutations of the sets of black and white vertices of $T$ respectively, such that:
    \begin{itemize}
        \item all cycles of $\sigma_b$ and $\sigma_w$ have odd length;
        \item each cycle carries a sign, either $+$ or $-$;
        \item the cycles of $\sigma_b$ (respectively $\sigma_w$) are labeled from $1$ to $|\sigma_b|$ (respectively $|\sigma_w|$), where $|\sigma|$ denotes the number of cycles in a permutation $\sigma$.
    \end{itemize}
    For $k,l,m,n \ge 1$, let $\mathcal{CT}_{k,l,m,n}$ be the set of C-decorated $(k,l)$-trees such that $|\sigma_b|=m$ and $|\sigma_w|=n$.
    
    For a finite set $\mathcal{A}$ we denote by $n\mathcal{A}$ the set made of $n$ disjoint copies of $\mathcal{A}$.
    
    \begin{theorem}[Chapuy, F\'eray, Fusy]
    \label{thm:CFF}
    For all $g\ge 0, k,l \ge 1$ there is a bijection 
    \[2^{k+l+2g}\mathcal{E}_{g,k,l}^{*,root} \simeq \bigsqcup_{\substack{g_1+g_2=g\\ g_1,g_2\ge 0}} \mathcal{CT}_{k+2g_1,l+2g_2,k,l},\]
    such that the underlying graph of each ribbon graph can be obtained from the corresponding tree by merging into a single vertex the vertices in each cycle of $\sigma_b$ and $\sigma_w$. The label of each cycle coincides with the label of the corresponding vertex.
    \end{theorem}
    
    Let $g_1, g_2 \ge 0$ be such that $g_1+g_2=g$ and let $b=(b_1,\ldots,b_k)$, $w=(w_1,\ldots,w_l)$ be tuples of non-negative integers such that $b_1+\ldots b_k = g_1, w_1+\ldots+w_l=g_2$. Denote by $\mathcal{CT}_{k+2g_1,l+2g_2,k,l}(b,w)$ the subset of C-decorated trees $(T,\sigma_b,\sigma_w) \in \mathcal{CT}_{k+2g_1,l+2g_2,k,l}$ such that the (labeled) cycles of $\sigma_b$ and $\sigma_w$ have sizes $2b_1+1, \ldots, 2b_k+1$ and $2w_1+1, \ldots, 2w_l+1$ respectively.
    
    Denote also by $\mathcal{E}_{g,k,l}^{*,root}(b,w)$ the subset of $2^{k+l+2g}\mathcal{E}_{g,k,l}^{*,root}$ which corresponds to $\mathcal{CT}_{k+2g_1,l+2g_2,k,l}(b,w)$ via the bijection of Theorem~\ref{thm:CFF}, so that 
    \[2^{k+l+2g}\mathcal{E}_{g,k,l}^{*,root} = \bigsqcup_{\substack{g_1+g_2=g\\ g_1,g_2\ge 0}} \bigsqcup_{\substack{b_1+\ldots b_k = g_1\\ w_1+\ldots+w_l=g_2}} \mathcal{E}_{g,k,l}^{*,root}(b,w).\]
    
    \begin{lemma}
    \label{lem:ribbonGraphToTree}
    There is a bijection
    \[\prod_{i=1}^k (2b_i+1) \prod_{j=1}^l (2w_j+1) \mathcal{E}_{g,k,l}^{*,root}(b,w) \simeq 2^{k+l} \mathcal{E}_{0,k+2g_1,l+2g_2}^{*,root}.\]
    In addition, the underlying graph of the ribbon graph can be obtained from the corresponding tree by merging into a single vertex:
    \begin{itemize}
        \item for each $i=1,\ldots,k$, black vertices with labels $\sum_{r=1}^{i-1} (2b_r+1) + 1, \ldots,  \sum_{r=1}^{i} (2b_r+1)$ to get the black vertex of $G$ with label $i$;
        \item for each $j=1,\ldots,l$, white vertices with labels $\sum_{r=1}^{j-1} (2w_r+1) + 1, \ldots,  \sum_{r=1}^{j} (2w_r+1)$ to get the white vertex of $G$ with label $j$.
    \end{itemize}
    \end{lemma}
    \begin{proof}
    Consider $G \in \mathcal{E}_{g,k,l}^{*,root}(b,w)$. Let $(T,\sigma_b,\sigma_w) \in \mathcal{CT}_{k+2g_1, l+2g_2, k, l}(b,w)$ be the corresponding C-decorated tree. One can associate to $T$ a family of rooted labeled trees from $\mathcal{E}_{0,k+2g_1,l+2g_2}^{*,root}$ by labelling the vertices of $T$ in such a way that the first cycle of $\sigma_b$ is $(1, 2, \ldots, 2b_1+1)$, the second cycle is $(2b_1+2,\ldots, 2b_1+2b_2+2)$, etc., and similarly for $\sigma_w$; then forgetting both the signs of cycles of $\sigma_b$, $\sigma_w$ and the permutations themselves (see Figure \ref{fig:graph-to-tree}). This can be done in $\prod_{i=1}^k (2b_i+1) \prod_{j=1}^l (2w_j+1)$ ways, since it is enough to choose in each cycle the vertex which will have the minimal label. 

    All the rooted labeled trees we get by the procedure described above will actually constitute the whole set $\mathcal{E}_{0,k+2g_1,l+2g_2}^{*,root}$ and each tree $T' \in \mathcal{E}_{0,k+2g_1,l+2g_2}^{*,root}$ will be obtained $2^{k+l}$ times, because to recover a C-decorated tree from $T'$ one firstly recovers the cycles of $\sigma_b$, $\sigma_w$ and their labels from the labels of the vertices of $T'$, but then one has to choose the signs of $k+l$ cycles of $\sigma_b$ and $\sigma_w$. Then one recovers $G \in \mathcal{E}_{g,k,l}^{*,root}(b,w)$ via the bijection of Theorem~\ref{thm:CFF}. This gives the desired bijection. The statement about the underlying graph of $G$ follows from the construction and from Theorem~\ref{thm:CFF}.
    \end{proof}
    
    \begin{figure}
        \centering
        \includegraphics[width=0.9\textwidth]{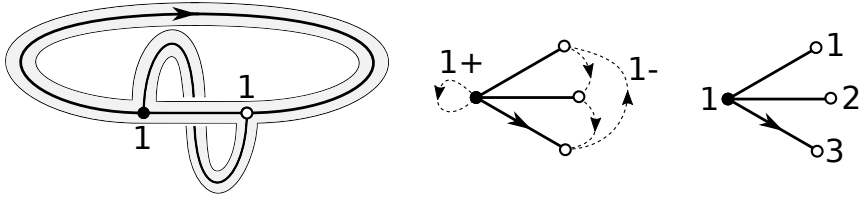}
        \caption{A ribbon graph from $\mathcal{E}_{1,1,1}^{*,root}$, one of its corresponding C-decorated trees from $\mathcal{CT}_{1, 3, 1, 1}$, and one of the rooted labeled trees from $\mathcal{E}_{0,1,3}^{*,root}$ corresponding to this C-decorated tree. The root edge is marked with an arrow.}
        \label{fig:graph-to-tree}
    \end{figure}
    
    We are now ready to prove Theorem~\ref{thm:explicitPolyWall}.
    
    \begin{proof}[Proof of Theorem~\ref{thm:explicitPolyWall}]
    Fix $W \in \overline{\mathcal{W}_{k,l}}$. By Remark~\ref{rmk:topDegreeTerm}, for every connected component $C$ of $H^+_{k,l} \cap W^\circ$, the top-degree term of the polynomial which gives $\mathcal{P}^g_{k,l}(L,L')$ for $(L,L') \in C \cap (\mathbb{Z}^k \times \mathbb{Z}^l)$ is given by (\ref{eq:top-degree-term}). We will show that these top-degree terms for different connected components coincide by giving an explicit expression for them, which will not depend on the connected component.
    
    Any $G \in \mathcal{E}_{g,k,l}^*$ can be rooted on any of its $(k+l+2g-1)$ edges. However, due to automorphisms, this makes $\frac{(k+l+2g-1)}{|\operatorname{Aut}(G)|}$ different rooted ribbon graphs. Hence (\ref{eq:top-degree-term}) is equal to 
    \begin{equation}
    \label{eq:rootedSum}
        (k+l+2g-1)^{-1} \cdot \sum_{G \in \mathcal{E}_{g,k,l}^{*,root}} \prod_{e \in B(G)} \mathbf{1}_{f_e(L,L')>0} \cdot \operatorname{Vol} M_G(L,L').
    \end{equation}
    
    Fix $g_1, g_2 \ge 0$ with $g_1+g_2=g$ and $b=(b_1,\ldots,b_k)$, $w=(w_1,\ldots,w_l)$ tuples of non-negative integers such that $b_1+\ldots + b_k = g_1, w_1+\ldots+w_l=g_2$. Let $G \in \mathcal{E}_{g,k,l}^{*,root}(b,w)$ and let $T \in \mathcal{E}_{0,k+2g_1,l+2g_2}^{*,root}$ be one of the rooted labeled trees corresponding to $G$ via the bijection of Lemma~\ref{lem:ribbonGraphToTree}.
    
    From Lemma~\ref{lem:ribbonGraphToTree} we know that the underlying graph of $G$ can be obtained from $T$ by merging its vertices in the corresponding groups. It means that one can choose an identification of the edges of $G$ with the edges of $T$ in such a way that an edge in $T$ joining vertices from two groups is identified with an edge of $G$ joining vertices that were merged from these two groups. Choose any such identification and, for each $e \in E(G)$, let $\widehat{e}\in E(T)$ be the corresponding edge in $T$.
    
    \begin{equation}
    \label{eq:diagram}
        \begin{tikzcd}
            \mathbb{R}^{E(G)} \arrow{r}{F} \arrow{d}{\operatorname{vp}_G} &  \mathbb{R}^{E(T)} \arrow{d}{\operatorname{vp}_T}\\
            \mathbb{R}^k \times \mathbb{R}^l & \arrow{l} \mathbb{R}^{k+2g_1} \times \mathbb{R}^{l+2g_2}
        \end{tikzcd}
    \end{equation}
    
    Consider the diagram (\ref{eq:diagram}). The identification of edges gives a linear isomorphism $F:\mathbb{R}^{E(G)} \rightarrow \mathbb{R}^{E(T)}$ between the spaces of weight functions on $G$ and $T$. Let $w_e, e \in E(G)$ and $w_{\widehat{e}}, e \in E(G)$ be the standard coordinates on these spaces. Let also $L, L'$ and $x, y$ be the coordinates on the spaces of vertex perimeters $\mathbb{R}^k \times \mathbb{R}^l$ and $\mathbb{R}^{k+2g_1} \times \mathbb{R}^{l+2g_2}$ respectively. Note that by Lemma~\ref{lem:metrics_on_trees}, since $T$ is a tree, $\operatorname{vp}_T$ is an isomorphism. The composition $\operatorname{vp}_G \circ F^{-1} \circ \operatorname{vp}_T^{-1}$ is given by
    \begin{equation}
    \label{eq:affineSubspaceEquations}
    \begin{cases}
    L_1=x_1+\ldots+x_{2b_1+1},\\
    L_2=x_{2b_1+2}+\ldots+x_{2b_1+2b_2+2},\\
    \dotfill\\
    L'_1=y_1+\ldots+y_{2w_1+1},\\
    L'_2=y_{2w_1+2}+\ldots+y_{2w_1+2w_2+2},\\
    \dotfill
    \end{cases}
    \end{equation}
    which follows from the construction of $G$ by merging the vertices of $T$.
    
    Denote $V(L,L') = \operatorname{vp}_T \circ F \circ \operatorname{vp}_G^{-1}(L,L')$. It is an affine subspace of $\mathbb{R}^{k+2g_1} \times \mathbb{R}^{l+2g_2}$ defined by equations (\ref{eq:affineSubspaceEquations}). We claim that
    \begin{equation}
    \label{eq:longEquation}
        \prod_{e \in B(G)} \mathbf{1}_{f_e(L,L')>0} \cdot \operatorname{Vol} M_G(L,L') = \int _{V(L,L')} \prod_{e \in E(T)} \mathbf{1}_{f_{e}(x, y)>0} \ d\operatorname{Vol}(x, y).
    \end{equation}
    Indeed, by the definition of $M_G(L,L')$ the left-hand side of (\ref{eq:longEquation}) is equal to
    \[
    \prod_{e \in B(G)} \mathbf{1}_{f_e(L,L')>0} \cdot \operatorname{Vol} \left[\operatorname{vp}_G^{-1}(L,L') \cap \{w_e>0, e \in E(G)\setminus B(G)\}\right].
    \]
    Since both $\operatorname{vp}_T$ and $F$ are invertible integral linear transformations, they both preserve integer lattices on affine subspaces, hence also the corresponding volumes. Hence we can apply $\operatorname{vp}_T \circ F$ to the expression inside $\operatorname{Vol}$:
    \[
    \prod_{e \in B(G)} \mathbf{1}_{f_e(L,L')>0} \cdot \operatorname{Vol} \left[\operatorname{vp}_{T} \circ F (\operatorname{vp}_G^{-1}(L,L')) \cap \operatorname{vp}_{T} \circ F (\{w_e>0, e \in E(G)\setminus B(G)\})\right].
    \]
    Recall that $\operatorname{vp}_T \circ F \circ \operatorname{vp}_G^{-1}(L,L') = V(L,L')$ by definition. In addition, $w_e>0 \iff w_{\widehat{e}}>0$, and by Lemma~\ref{lem:bridge_length_formula} the edge weights in $T$ are uniquely determined by vertex perimeters: $w_{\widehat{e}} = f_{\widehat{e}}(x, y)$.
    Hence the last expression is equal to
    \begin{align*}
        &\prod_{e \in B(G)} \mathbf{1}_{f_e(L,L')>0} \cdot \operatorname{Vol} \Bigl[ V(L,L') \cap \{f_{\widehat{e}}(x, y)>0, e \in E(G)\setminus B(G)\} \Bigr]\\
        &=\prod_{e \in B(G)} \mathbf{1}_{f_e(L,L')>0} \cdot \int _{V(L,L')} \prod_{e \in E(G)\setminus B(G)} \mathbf{1}_{f_{\widehat{e}}(x, y)>0} \ d\operatorname{Vol}(x, y)
    \end{align*}
     
     For $(x, y) \in V(L,L')$ and $e \in B(G)$ we have $f_e(L,L') = w_e = w_{\widehat{e}} = f_{\widehat{e}}(x, y)$, so we finally get
     \[
     \int _{V(L,L')} \prod_{e \in E(G)} \mathbf{1}_{f_{\widehat{e}}(x, y)>0} \ d\operatorname{Vol}(x, y) =\int _{V(L,L')} \prod_{e \in E(T)} \mathbf{1}_{f_{e}(x, y)>0} \ d\operatorname{Vol}(x, y),
     \]
     which is the right-hand side of (\ref{eq:longEquation}).
    
    Now, summing the equality (\ref{eq:longEquation}) over all $G \in \prod_{i=1}^k (2b_i+1) \prod_{j=1}^l (2w_j+1) \mathcal{E}_{g,k,l}^{*,root}(b,w)$ and applying Lemma~\ref{lem:ribbonGraphToTree} we get
    \begin{align*}
        &\prod_{i=1}^k (2b_i+1) \prod_{j=1}^l (2w_j+1) \sum_{G \in \mathcal{E}_{g,k,l}^{*,root}(b,w)} \prod_{e \in B(G)} \mathbf{1}_{f_e(L,L')>0} \cdot \operatorname{Vol} M_G(L,L')\\
        &= 2^{k+l} \int_{V(L,L')} \left( \sum_{T \in \mathcal{E}_{0,k+2g_1,l+2g_2}^{*,root}} \prod_{e \in E(T)} \mathbf{1}_{f_{e}(x, y)>0} \right) \ d\operatorname{Vol}(x, y).
    \end{align*}
    Since each tree $T \in \mathcal{E}_{0,k+2g_1,l+2g_2}^*$ can be rooted on any of its $(k+l+2g-1)$ edges and $T$ has no non-trivial automorphisms (as noted at the beginning of Section~\ref{subsec:proofThmTrees}), the last expression is equal to 
    \begin{equation*}
        2^{k+l} (k+l+2g-1) \int_{V(L,L')} \left( \sum_{T \in \mathcal{E}_{0,k+2g_1,l+2g_2}^*} \prod_{e \in E(T)} \mathbf{1}_{f_{e}(x, y)>0} \right) \ d\operatorname{Vol}(x, y).
    \end{equation*}
    
    Note that from (\ref{eq:affineSubspaceEquations}) it follows that when $(L,L') \in H^+_{k,l} \cap W^\circ$, the generic point $(x, y)$ of the affine subspace $V(L,L')$ lies in $(U_{W,b,w})^\circ$ for some corresponding $U_{W,b,w} \in  \overline{\mathcal{W}_{k+2g_1,l+2g_2}}$. Consider the sum inside the integral in the last expression. It was proven in Section~\ref{subsec:proofThmTrees} that this sum is constant on $H^+_{k+2g_1,l+2g_2} \cap (U_{W,b,w})^\circ$ with the corresponding value $P^0_{U_{W,b,w}}$, and is zero outside of $H^+_{k+2g_1,l+2g_2}$. So the last expression is equal to
    \begin{align*}
        & 2^{k+l} (k+l+2g-1) \int_{V(L,L') \cap H^+_{k+2g_1,l+2g_2}} P^0_{U_{W,b,w}} \ d\operatorname{Vol}(x, y)\\
        & = 2^{k+l} (k+l+2g-1) \cdot P^0_{U_{W,b,w}} \cdot \operatorname{Vol}(V(L,L') \cap H^+_{k+2g_1,l+2g_2})\\
        & = 2^{k+l} (k+l+2g-1) \cdot P^0_{U_{W,b,w}} \cdot \prod_{i=1}^k \frac{L_i^{2b_i}}{(2b_i)!} \cdot \prod_{j=1}^l \frac{{L'_j}^{2w_j}}{(2w_j)!}.
    \end{align*}
    Thus
    \begin{align*}
        &\sum_{G \in \mathcal{E}_{g,k,l}^{*,root}(b,w)} \prod_{e \in B(G)} \mathbf{1}_{f_e(L,L')>0} \cdot \operatorname{Vol} M_G(L,L') \\
        &= 2^{k+l} (k+l+2g-1) \cdot P^0_{U_{W,b,w}} \cdot \prod_{i=1}^k \frac{L_i^{2b_i}}{(2b_i+1)!} \cdot \prod_{j=1}^l \frac{{L'_j}^{2w_j}}{(2w_j+1)!}.
    \end{align*}
    Summing this over all $g_1,g_2$ and all $b,w$ we get 
    \begin{align*}
        &2^{k+l+2g} \sum_{G \in \mathcal{E}_{g,k,l}^{*,root}} \prod_{e \in B(G)} \mathbf{1}_{f_e(L,L')>0} \cdot \operatorname{Vol} M_G(L,L') \\
        &= 2^{k+l} (k+l+2g-1) \cdot \sum_{\substack{b_1+\ldots+b_k+w_1+\ldots+w_l=g\\b_i,w_i \ge 0}} P^0_{U_{W,b,w}} \cdot \prod_{i=1}^k \frac{L_i^{2b_i}}{(2b_i+1)!} \cdot \prod_{j=1}^l \frac{{L'_j}^{2w_j}}{(2w_j+1)!}.
    \end{align*}
    Taking (\ref{eq:rootedSum}) into account we finally get the expression (\ref{eq:localPolyExplicit}) which does not depend on the connected component of $H^+_{k,l} \cap W^\circ$.
    \end{proof}
    
    \section{Coefficients of top-degree terms}
    \label{sec:coefficients}
    In Section~\ref{subsec:proofRecursion} we use the combinatorial interpretation of the coefficients of the top-degree terms of the counting functions to find a recursion for them. This recursion is then shown to imply Theorem~\ref{thm:polys}. In Section~\ref{subsec:proof_main_theorem} we deduce the main Theorem~\ref{thm:main_theorem} from all of the previously obtained properties of the counting functions.
    
    \subsection{Recursion for the values of \texorpdfstring{$P^0_W$}{P0W} and proof of Theorem~\ref{thm:polys}}
    \label{subsec:proofRecursion}
    
    Let $n, k,l \ge 1$ and let $b_1+\ldots+b_n=k$, $w_1+\ldots+w_n=l$. Denote by $W^{b_1,\ldots,b_n}_{w_1,\ldots,w_n}$ the subspace from  $\overline{\mathcal{W}_{k,l}}$ defined by the equations
    \begin{equation}
        \label{eq:wall_equations}
        \begin{aligned}
        L_1+\ldots+L_{b_1} &= L'_1+\ldots+L'_{w_1},\\
        L_{b_1+1}+\ldots+L_{b_1+b_2} &= L'_{w_1+1}+\ldots+L'_{w_1+w_2},\\
        &\dots\\
        L_{b_1+\ldots+b_{n-1}+1}+\ldots+L_{k} &= L'_{w_1+\ldots+w_{n-1}+1}+\ldots+L'_{l},
        \end{aligned}
    \end{equation}
    and denote by $p^{b_1,\ldots,b_n}_{w_1,\ldots,w_n}$ the unique value of $P^0_{W^{b_1,\ldots,b_n}_{w_1,\ldots,w_n}}$.
    
    It follows from Theorem~\ref{thm:explicitPolyWall} and its proof that the polynomial $P^g_{V_n}$ with $V_n=\{L_1=L'_1, \ldots, L_n=L'_n\} \in \overline{\mathcal{W}_{n,n}}$ is given by
    \begin{equation}
    \label{eq:PgVn-explicit}
        P^g_{V_n}(L,L) = 2^{-2g} \cdot \sum_{\substack{b_1+\ldots+b_n+w_1+\ldots+w_n=g\\b_i,w_i \ge 0}} p^{2b_1+1,\ldots,2b_n+1}_{2w_1+1,\ldots,2w_n+1} \cdot \prod_{i=1}^n \frac{L_i^{2(b_i+w_i)}}{(2b_i+1)!(2w_i+1)!}.
    \end{equation}
    
    Taking this into account, we see that to prove Theorem~\ref{thm:polys} we need to study the numbers $p^{b_1,\ldots,b_n}_{w_1,\ldots,w_n}$ and relations between them. By Remark~\ref{rmk:P0CountsPositiveTrees}, $p^{b_1,\ldots,b_n}_{w_1,\ldots,w_n}$ is the number of trees positive at $(L,L')$ for $(L,L') \in H^+_{k,l} \cap (W^{b_1,\ldots,b_n}_{w_1,\ldots,w_n})^\circ$. We will use this combinatorial interpretation of these numbers throughout this section.
    
    \begin{lemma}
    \label{lem:generic_value}
        For all $k,l\ge 1$ we have $p^k_l=(k+l-2)!$.
    \end{lemma}
    \begin{proof}
        Consider the point $(L,L') = (N, 1, \ldots, 1; \frac{N+k-1}{l}, \ldots, \frac{N+k-1}{l}) \in \mathbb{R}^k \times \mathbb{R}^l$ with $N\gg kl$. It is easy to check that it belongs to $H^+_{k,l} \cap (W^k_l)^\circ = H^+_{k,l} \cap (H_{k,l})^\circ$. Thus $p^k_l$ is the number of trees positive at $(L,L')$. These vertex perimeters force a particularly simple structure of the positive trees. Since $\frac{L+k-1}{l} \cdot (l-1) < L$, all of the $l$ white vertices must be adjacent to the black vertex with perimeter $L$. The remaining $k-1$ black vertices with perimeters 1 can be attached to the white vertices in an arbitrary manner. The total number of positive trees is then 
        \[(l-1)! \cdot l \cdot (l+1) \cdot \ldots \cdot (l+k-2) = (k+l-2)!.\]
    \end{proof}
    
    In general, the numbers $p^{b_1,\ldots,b_n}_{w_1,\ldots,w_n}$ can be computed recursively using the following relation.
    
    \begin{prop}
    \label{prop:generic_values_recursion}
    Let $n,k,l\geq 1$ and let $b_i, w_i \ge 1$ such that $b_1+\ldots+b_n = k$ and $w_1+\ldots+w_n = l$. Then
    \begin{equation}
    \label{eq:pNumbersRecursion}
    (k+l-2)! = p^{b_1,\ldots,b_n}_{w_1,\ldots,w_n} + \sum_{t=2}^n \frac{(k+l-2)_{t-2}}{t!} \sum_{I_1, \ldots I_t} \prod_{j=1}^t \left(\sum_{i\in I_j} (b_i+w_i) - 1\right) p^{b_{I_j}}_{w_{I_j}},
    \end{equation}
    where $(x)_t :=x(x-1)\cdots (x-t+1)$ is the falling factorial ($(x)_0:=1$); the second sum is over all partitions of $\{1,\ldots,n\}$ into $t$ non-empty labeled sets $I_1,\ldots,I_t$; $b_{I_j}$ denotes $\{b_i\}_{i\in I_j}$, and analogously for $w_{I_j}$.
    \end{prop}
    \begin{proof}
        Let $(L_1, \ldots, L_k; L'_1, \ldots, L'_l)$ be a point in $H^+_{k,l} \cap (W^{b_1,\ldots,b_n}_{w_1,\ldots,w_n})^\circ$. For $i=1,\ldots, n$, let $A_{b,i} = \{b_1+\ldots+b_{i-1}+1,\ldots,b_1+\ldots+b_i\}$ and $A_{w,i} = \{w_1+\ldots+w_{i-1}+1,\ldots,w_1+\ldots+w_i\}$. The defining equations (\ref{eq:wall_equations}) of the wall $W^{b_1,\ldots,b_n}_{w_1,\ldots,w_n}$ can now be rewritten as 
        \begin{equation}
            \label{eq:wall_equations_2}
            \sum_{j\in A_{b,i}} L_j = \sum_{j \in A_{w,i}} L'_j, \ i=1,\ldots,n.
        \end{equation}
        Consider now a path of the form \[(L_1+(k-1)\varepsilon, L_2-\varepsilon, \ldots, L_k-\varepsilon; L'_1, \ldots, L'_k),\  \varepsilon \rightarrow 0+.\] 
        While $\varepsilon>0$ and $\varepsilon$ is sufficiently small, the point in question lies in $H^+_{k,l} \cap (H_{k,l})^\circ$  and the number of trees positive at this point is equal to $(k+l-2)!$ by Lemma~\ref{lem:generic_value}. When $\varepsilon=0$, the number of trees positive at this point is equal to $p^{b_1,\ldots,b_n}_{w_1,\ldots,w_n}$ by definition. Hence we can deduce the desired formula (\ref{eq:pNumbersRecursion}) if we enumerate the trees that cease to be positive when $\varepsilon \rightarrow 0+$.

        Let $\mathcal{D}$ be the set of these degenerate trees, and consider a tree from $\mathcal{D}$. It consists of several positive trees connected by zero-weight edges (Figure \ref{fig:degenerate_tree}, left). Applying the edge weight formula of Lemma~\ref{lem:bridge_length_formula} to a zero-weight edge, we obtain a linear relation between the vertex perimeters of the form $\sum_{i \in A_b} L_i = \sum_{j \in A_w} L'_j$. Since $(L_1, \ldots, L_k; L'_1, \ldots, L'_l) \in (W^{b_1,\ldots,b_n}_{w_1,\ldots,w_n})^\circ$, the only such linear relations possible are given by the equations (\ref{eq:wall_equations_2}) and the sums of such equations. Hence $A_b=\bigcup_{i\in I} A_{b,i}, A_w=\bigcup_{i\in I} A_{w,i}$ for some index set $I\subset\{1,\ldots,n\}$. This implies that in every constituent positive tree the set of vertex labels is of the form $\bigcup_{i\in I} A_{b,i}$ for black vertices and  $\bigcup_{i\in I} A_{w,i}$ for white vertices, for some index set $I$. 

        \begin{figure}
            \centering
            \includegraphics[width=0.8\textwidth]{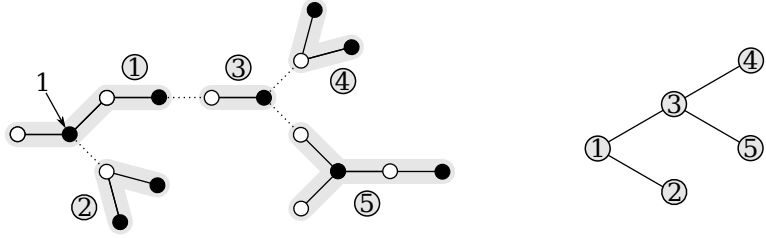}
            \caption{Left: a degenerate tree $G \in \mathcal{D}$ with $t=5$ constituent positive trees, which are shaded in grey. Their labels are circled. Zero-weight edges are dotted. Black vertex number 1 is in the positive tree number 1. Note that for every zero-weight edge $e$ in $G$, the black extremity of $e$ is in the same connected component of $G-e$ as the black vertex number $1$.  Right: the corresponding tree $T$.}
            \label{fig:degenerate_tree}
        \end{figure}
        
        In the desired formula (\ref{eq:pNumbersRecursion}): $2 \leq t \leq n$ stands for the number of constituent positive trees, which we label by numbers from $1$ to $t$ for convenience; the factor $1/t!$ accounts for the arbitrariness of the numbering of these trees; $I_1, \ldots, I_t$ are such that the vertex labels of the $j$-th positive tree are $\bigcup_{i\in I_j}A_i$ and $\bigcup_{i\in I_j} B_i$; $\prod_{j=1}^t p^{a_{I_j}}_{b_{I_j}}$ is the number of possible choices of positive trees themselves. We now fix $t$, the partition $I_1,\ldots, I_t$, and a choice of $t$ positive trees, and we will count the number of ways to connect them with zero-weight edges to form a degenerate tree from $\mathcal{D}$. We claim that this count gives the remaining factor 
        \begin{equation}
            (k+l-2)_{t-2} \cdot \prod_{j=1}^t \left(\sum_{i\in I_j} (b_i+w_i) - 1\right).
        \end{equation}

        Consider a zero-weight edge $e$ of a degenerate tree $G$ from $\mathcal{D}$. We claim that the black extremity of $e$ is in the same connected component of $G-e$ as the black vertex number 1. Indeed, if it were not true, then by Lemma~\ref{lem:bridge_length_formula} the weight of this edge when $\varepsilon>0$ would be equal to 
        \[\sum_{i\in I} \left(\sum_{j\in A_{b,i}} (L_j - \varepsilon) - \sum_{j\in A_{w,i}} L'_j\right) = - \left(\sum_{i\in I} b_i \right)\cdot \varepsilon <0,\]
        for some index set $I\subset\{1,\ldots,n\}$, a contradiction. 

        Without loss of generality, assume that the black vertex number 1 is in the positive tree number 1 (the resulting count of degenerate trees will not depend on this choice). The above argument implies two conditions.
        
        \emph{Condition 1.} The zero-weight edges that are incident to the positive tree number 1 are incident only to its black vertices.
        
        \emph{Condition 2.} For each positive tree number $j\geq 2$, there is exactly one incident zero-weight edge which is incident to its white vertex, and several (maybe none) incident zero-weight edges which are incident to its black vertices.

        To each way of connecting the positive trees with zero-weight edges we put into correspondence a non-plane labeled tree $T$ with $t$ vertices, where the vertex number $j$ of $T$ corresponds to the positive tree number $j$, and the edges of $T$ correspond to zero-weight edges joining the positive trees (Figure \ref{fig:degenerate_tree}, right). 
        
        Let $c_j = \sum_{i \in I_j} (b_i+w_i)$ be the number of vertices in the positive tree number $j$. Consider a tree $T$ such that the degree of the vertex number $j$ is equal to $d_j$. It follows from the two conditions above that there are
        \[(c_1 - 1)^{(d_1)} \cdot \prod_{j=2}^t (c_j - 1) (c_j - 1)^{(d_j - 1)}\] 
        ways of connecting the positive trees with zero-weight edges to form a degenerate tree from $\mathcal{D}$, and which correspond to $T$; here $x^{(n)}=x(x+1)\cdots (x+n-1)$ is the rising factorial ($x^{(0)}:=1$). Indeed, a bipartite plane tree with $N$ vertices has $N-1$ corners around black (white) vertices where we can glue an edge. If we have glued an edge to a black (white) vertex, the number of corners available for gluing around black (white) vertices increases by one. This gives the rising factorials in the formula. 

        It is well known that the number of non-plane labeled trees $T$ on $t$ vertices with the degree of the vertex number $i$ equal to $d_i$ is 
        \[\binom{t-2}{d_1 -1, \ldots, d_t-1}\]
        provided $d_1 + \ldots + d_t = 2t-2$. Hence the missing factor in our desired formula is equal to 
        \begin{align*}
            &\prod_{j=1}^t (c_j - 1) \cdot 
            \sum_{\substack{d_1+\ldots+d_t=2t-2\\ d_i \geq 1}}
            \binom{t-2}{d_1 -1, \ldots, d_t-1}  
            c_1^{(d_1 - 1)} \prod_{j=2}^t (c_j-1)^{(d_j - 1)}\\
            &= \prod_{j=1}^t (c_j - 1) \cdot (t-2)! \cdot \sum_{\substack{d'_1+\ldots+d'_t=t-2\\ d'_i \geq 0}} \binom{c_1-1+d'_1}{d'_1} \prod_{j=2}^t \binom{c_j-2+d'_j}{d'_j}\\
            &=\prod_{j=1}^t (c_j - 1) \cdot (t-2)! \cdot \binom{k+l-2}{t-2} = \prod_{j=1}^t (c_j - 1) \cdot (k+l-2)_{t-2}.
        \end{align*}
        In the first equality we perform the change of variables $d'_i=d_i-1$. The second equality follows from the following general identity, valid for $n,m \in \mathbb{N} \cup \{0\}$, $x_1, \ldots, x_n \in \mathbb{R}$:
        \[ \sum_{\substack{d_1+\ldots+d_n=m\\ d_i \in \mathbb{N}\cup \{0\}}} \binom{x_1+d_1}{d_1} \cdots \binom{x_n+d_n}{d_n} = \binom{\sum_{i=1}^n x_i + (n - 1) + m}{m},\]
        which can be obtained by extracting the coefficient of $z^m$ on both sides of the identity 
        \[(1-z)^{-(x_1+1)} \cdots (1-z)^{-(x_n+1)} = (1-z)^{-(x_1 + \ldots + x_n + (n-1) + 1)}.\]
        We also use the fact that $c_1+\ldots+c_t = k+l$.
\end{proof}

    \begin{corollary}
    \label{cor:dependenceOnSum}
    The value of $p^{b_1,\ldots,b_n}_{w_1,\ldots,w_n}$ depends only on $b_i+w_i, i=1,\ldots,n$.
    \end{corollary}
    \begin{proof}
    The proof is by induction on $n$. The base case $n=1$ follows from the explicit formula of Lemma~\ref{lem:generic_value}. The induction step follows from the formula of Proposition~\ref{prop:generic_values_recursion}, because the left hand side is just $\left(\sum_i (b_i+w_i) - 2\right)!$, and the big sum on the right hand side depends only on $b_i+w_i$ by the induction hypothesis (the numbers $p^{b_{I_j}}_{w_{I_j}}$ have strictly less then $n$ indices).
    \end{proof}
    
    \begin{proof}[Proof of Theorem~\ref{thm:polys}]
    Denote by $p_{s_1,\ldots,s_n}$ the common value of $p^{b_1,\ldots,b_n}_{w_1,\ldots,w_n}$ with $b_i+w_i=s_i$, which is well-defined by Corollary~\ref{cor:dependenceOnSum}. Then it follows from (\ref{eq:PgVn-explicit}) that 
    \begin{align*}
        & P^g_{V_n}(L,L) = 2^{-2g} \cdot \sum_{\substack{s_1+\ldots+s_n=g\\s_i \ge 0}} p_{2s_1+2,\ldots,2s_n+2} \cdot \prod_{i=1}^n \left( \sum_{b_i+w_i=s_i} \frac{L_i^{2s_i}}{(2b_i+1)!(2w_i+1)!}\right)\\
        & = 2^{-2g} \cdot \sum_{\substack{s_1+\ldots+s_n=g\\s_i \ge 0}} p_{2s_1+2,\ldots,2s_n+2} \cdot \prod_{i=1}^n \left(L_i^{2s_i} \cdot \frac{2^{2s_i+1}}{(2s_i+2)!}\right)\\
        & = 2^n \cdot \sum_{\substack{s_1+\ldots+s_n=g\\s_i \ge 0}} p_{2s_1+2,\ldots,2s_n+2} \prod_{i=1}^n \frac{L_i^{2s_i}}{(2s_i+2)!}\\
        & = 2^n \cdot \sum_{\substack{s_1+\ldots+s_n=g\\s_i \ge 1}} p_{2s_1,\ldots,2s_n} \prod_{i=1}^n \frac{L_i^{2s_i-2}}{(2s_i)!}.
    \end{align*}
    In the second equality we have used the fact that $\sum_{b_i+w_i=s_i} \binom{2s_i+2}{2b_i+1} = 2^{2s_i+1}$ and so $\sum_{b_i+w_i=s_i} \frac{1}{(2b_i+1)!(2w_i+1)!} = \frac{2^{2s_i+1}}{(2s_i+2)!}$. The last equality is just a change of variables $s_i=s_i+1$.
    
    We have obtained the desired expression for $P^g_{V_n}$. Now we have to show that the generating function $\mathcal{T}$ of the numbers $p_{s_1,\ldots,s_n}$ satisfies the relation (\ref{eq:relation_multivariate}).
    
    The recurrence relation (\ref{eq:pNumbersRecursion}) can be rewritten with the new notation $s_i=b_i+w_i$, $s=s_1+\ldots+s_n$ as 
    \begin{equation}
    \label{eq:pNumbersRecursionBis}
        (s -2)! = p_{s_1,\ldots,s_n} + \sum_{t=2}^n \frac{(s-2)_{t-2}}{t!} \cdot \sum_{I_1, \ldots I_t} \prod_{j=1}^t \left(\sum_{i\in I_j} s_i - 1\right)\cdot p_{s_{I_j}},
    \end{equation}
    where the second sum is over all partitions of $\{1,\ldots,n\}$ into $t$ non-empty labeled sets $I_1,\ldots,I_t$, and $s_{I_j}$ denotes $\{s_i\}_{i\in I_j}$.
    
    Now multiply (\ref{eq:pNumbersRecursionBis}) by $s(s-1) \frac{1}{n!} t_{s_1} \cdots t_{s_n}$ and sum over all $n \ge 1$ and all $s_1, \ldots, s_n \ge 2$ such that $s_1+\ldots+s_n=s$. The left hand side becomes
    \begin{equation}
    \label{eq:lhsSummation}
        s! \cdot \sum_{n \ge 1} \frac{1}{n!} \sum_{\substack{s_1+\ldots+s_n=s\\ s_i \ge 2}} t_{s_1} \cdots t_{s_n} = s!\cdot [t^s] \exp\left(\sum_{i\geq 2} t_i t^i\right).
    \end{equation}
    The right-hand side becomes
    \begin{equation*}
        \sum_{n \ge 1} \frac{1}{n!} \sum_{\substack{s_1+\ldots+s_n=s\\ s_i \ge 2}} \sum_{t=1}^n \binom{s}{t} \sum_{I_1, \ldots I_t} \prod_{j=1}^t \left(\sum_{i\in I_j} s_i - 1\right) p_{s_{I_j}} \prod_{i \in I_j} t_{s_i}.
    \end{equation*}
    Changing the order of summation, this is equal to
    \begin{equation*}
        \sum_{t=1}^s \binom{s}{t} \sum_{n \ge t} \frac{1}{n!} \sum_{\substack{s_1+\ldots+s_n=s\\ s_i \ge 2}} \sum_{I_1, \ldots I_t} \prod_{j=1}^t \left(\sum_{i\in I_j} s_i - 1\right) p_{s_{I_j}} \prod_{i \in I_j} t_{s_i}.
    \end{equation*}
    Denote $\Sigma_j = \sum_{i\in I_j} s_i$ and $n_j = |I_j|$. Let also $I_j = \{i^j_1,\ldots, i^j_{n_j}\}$ with $i^j_1 < \ldots < i^j_{n_j}$.
    
    Fix $s$, $t$ and $n$. To every pair consisting of a composition $s_1+\ldots+s_n=s$ and a labeled partition $I_1 \sqcup \ldots \sqcup I_t = \{1,\ldots,n\}$ we put into correspondence a composition $\Sigma_1+\ldots+\Sigma_t=s$ and $t$ compositions  $s_{i^j_1} + \ldots + s_{i^j_{n_j}} = \Sigma_j$. This correspondence is $\frac{n!}{n_1!\cdots n_t!}$-to-1, because the tuple $(s_1,\ldots,s_n)$ can be uniquely reconstructed if the partition $I_1 \sqcup \ldots \sqcup I_t$ is known, and there are $\frac{n!}{n_1!\cdots n_t!}$ ways to choose this partition. Hence the last sum can be rewritten as
    \begin{equation*}
        \sum_{t=1}^s \binom{s}{t} \sum_{\Sigma_1+\ldots+\Sigma_t=s} \prod_{j=1}^t (\Sigma_j-1) \left( \sum_{n_j\ge 1} \frac{1}{n_j!} \sum_{s_{i^j_1} + \ldots + s_{i^j_{n_j}} = \Sigma_j} p_{s_{i^j_1},\ldots, s_{i^j_{n_j}}} t_{s_{i^j_1}} \cdots t_{i^j_{n_j}} \right).
    \end{equation*}
    This last expression is equal to 
    \begin{equation}
    \label{eq:rhsSummation}
        [t^s]\mathcal{T}(t,t_2,t_3,\ldots)^s.
    \end{equation}
    Equating (\ref{eq:lhsSummation}) and (\ref{eq:rhsSummation}) we get relation (\ref{eq:relation_multivariate}) for $k\ge 2$. For $k=0, 1$ this relation can be easily verified from definitions.
    \end{proof}

    \subsection{Proof of main theorem}
    \label{subsec:proof_main_theorem}
    We start with the following elementary statement.
    
    \begin{lemma}
    \label{lem:asymptotics}
    Let $n\geq 1$ and $s_1,\ldots,s_n \in \mathbb{Z}_{>0}$. Then, as $N\rightarrow \infty$,
    \[\sum_{\substack{\sum_{i=1}^n h_i L_i \leq N\\ h_i, L_i \in \mathbb{Z}_{>0}}} L_1^{s_1} \cdots L_n^{s_n} \sim \frac{N^{s+2n}}{(s+2n)!} \cdot \prod_{i=1}^n (s_i ! \cdot \zeta(s_i+1)),\]
    where $s=s_1+\ldots+s_n$ and $\zeta$ is the Riemann zeta function.
    \end{lemma}
    
    By linearity, Lemma~\ref{lem:asymptotics} allows to compute the asymptotics of any such sum where the function being summed is a polynomial in the variables $L_1,\ldots, L_n$ divisible by $L_1\cdots L_n$. Only the terms of top degree contribute to the total asymptotics.
    
    The proof of Lemma~\ref{lem:asymptotics} can be found in the paper \cite{AEZ}, Lemma 3.7. The proof proceeds by approximating the properly normalized initial sum by an integral of a polynomial over the standard simplex. In particular, the statement also holds when the function being summed is a polynomial in the variables $L_1,\ldots, L_n$ divisible by $L_1\cdots L_n$, but only outside of a finite number of hyperplanes (measure zero set), where it is given by polynomials of at most the same degree (this last condition is sufficient to ensure that the term with the integral over the ``exceptional'' locus does not contribute to the asymptotics when $N \rightarrow \infty$).
    
    We are now ready to prove Proposition~\ref {prop:explicit_formulas_contributions} as well as the main Theorem~\ref{thm:main_theorem}.
    
    \begin{proof}[Proof of Proposition~\ref{prop:explicit_formulas_contributions}]
    Combining the formula (\ref{eq:contribution_via_counting}) and Proposition~\ref{prop:k-cyl_contribution}, we see that 
    \begin{multline*}
        \operatorname{Vol}_n(2g-2) = 2 \cdot 2g \cdot \frac{1}{n!} \cdot \\
        \lim_{N\rightarrow \infty} N^{-2g} \cdot \sum_{\substack{\sum_{i=1}^n h_i L_i \leq N\\ h_i, L_i \in \mathbb{Z}_{>0}}} L_1\cdots L_n \cdot \mathcal{P}^{g-n}_{n,n}(L_1,\ldots,L_n, L_1,\ldots,L_n).
    \end{multline*}
    By the remarks that follow Lemma~\ref{lem:asymptotics}, we can replace $\mathcal{P}^{g-n}_{n,n}$ by its top-degree term $P^{g-n}_{V_n}$. Then, substituting the explicit formula for the polynomial $P^{g-n}_{V_n}$ from Theorem~\ref{thm:polys}, and using Lemma~\ref{lem:asymptotics} we get
    \[\frac{2}{(2g-1)!} \cdot \frac{1}{n!} \cdot \sum_{\substack{s_1+\ldots+s_n=g\\ s_i\geq 1}} p_{2s_1,\ldots,2s_n} \frac{\zeta(2s_1)}{s_1}\cdots \frac{\zeta(2s_n)}{s_n},\]
    which is the first formula of Proposition~\ref{prop:explicit_formulas_contributions}. In particular, the limit in (\ref{eq:contribution_via_counting}) exists.
    
    The second formula follows from the fact that $\operatorname{Vol}_n(2g-2) = \frac{2(2\pi)^{2g}}{(2g-1)!}a_{g,n}$ by definition, and that for $s\in \mathbb{Z}_{>0}$ one has $\zeta(2s)=\frac{(-1)^{s+1}B_{2s} (2\pi)^{2s}}{2\cdot (2s)!}$, where $B_{2s}$ is the $2s$-th Bernoulli number.
    \end{proof}
    
    \begin{proof}[Proof of Theorem~\ref{thm:main_theorem}]
    It follows from the equality $\operatorname{Vol}_n(2g-2) = \frac{2(2\pi)^{2g}}{(2g-1)!}a_{g,n}$ and from the explicit expression for $\operatorname{Vol}_n(2g-2)$ from Proposition~\ref{prop:explicit_formulas_contributions} that $\mathcal{C}(2\pi t, u)$ is equal to $\mathcal{T}(t,t_2,t_3,\ldots)$ evaluated at $t_{2i}=\frac{\zeta(2i)}{i}u$ and $t_{2i+1}=0$ for $i\geq 1$, where $\mathcal{T}$ is defined in Theorem~\ref{thm:polys}. Then the relation (\ref{eq:relation_multivariate}) from Theorem~\ref{thm:polys} gives for $k=2g, g\geq 0$.
    \[\frac{1}{(2g)!} [t^{2g}] \mathcal{C}(2\pi t,u)^{2g} = [t^{2g}] \exp\left(u \cdot \sum_{i\geq 1} \frac{\zeta(2i)}{i} t^{2i}\right).\]
    The series inside the exponent can be rewritten in terms of logarithms (as noted in Lemma 3.8 of \cite{DGZZ-asymptoticGeometry}): expanding the definition of the zeta function and changing the order of summation we find that it is equal to 
    \[u \cdot \sum_{i\geq 1} \log\left(1-\frac{t^2}{i^2}\right),\]
    so the exponent is
    \[\left( \prod_{i\geq 1} \left(1-\frac{t^2}{i^2}\right) \right)^u = \left(\frac{\sin(\pi t)}{\pi t}\right)^u,\]
    where we have used the well-known product formula for the sine function. Replacing $2\pi t$ by $t$, we get the desired relation (\ref{eq:relation_bivariate}).
    \end{proof}
    
    
\newcommand{\etalchar}[1]{$^{#1}$}

\end{document}